\documentclass[11pt]{amsart}
\usepackage{amsthm,amssymb,amsmath}

\numberwithin{equation}{section}

\newtheorem{theorem}{Theorem}[section]

\newtheorem{lemma}[theorem]{Lemma}

\newtheorem{corollary}[theorem]{Corollary}
\newtheorem{proposition}[theorem]{Proposition}

\theoremstyle{definition}

\newtheorem{definition}[theorem]{\sc Definition}
\newtheorem{defin}[theorem]{\sc Definition}
\newtheorem{example}[theorem]{\sc Example}
\newtheorem{remark}[theorem]{\bf Remark}

\newcommand{\Imag}{{\rm Im~}}

\newcommand{\Id}{{\rm Id}}

\newcommand{\trig}{{\rm 0}}

\newcommand{\Hom}{{\rm Hom_{\mathcal O}}}

\newcommand{\Ker}{{\rm Ker~}}

\DeclareRobustCommand{\gobblefour}[4]{}
\newcommand*{\SkipTocEntry}{\addtocontents{toc}{\gobblefour}}

\include{diagxy}

\begin{document}

%
%
%

\title[Algebras of holomorphic functions]{Towards Oka-Cartan theory for algebras of holomorphic functions on coverings of Stein manifolds I}

\author{A.~Brudnyi}

\address{Department of Mathematics and Statistics, University of Calgary, 
Calgary, Canada}

\email{albru@math.ucalgary.ca}

\author{D.~Kinzebulatov}

\address{The Fields Institute, Toronto, Canada}

\email{dkinzebu@fields.utoronto.ca}



\begin{abstract}
We develop complex function theory within certain algebras of holomorphic functions
on coverings of Stein manifolds. This, in particular, includes the results on holomorphic
extension from complex submanifolds, corona type theorems, properties of divisors,
holomorphic analogs of the Peter-Weyl approximation theorem,
Hartogs type theorems, characterization of uniqueness sets.
The model examples of these algebras are: 

(1) Bohr's algebra of holomorphic almost
periodic functions on tube domains; 

(2) algebra of all fibrewise bounded holomorphic functions
(e.g., arising in the corona problem for $H^\infty$).  

Our approach is based on an extension of the classical Oka-Cartan theory to coherent-type sheaves
on the maximal ideal spaces of these algebras -- topological spaces having some features of complex manifolds.
\end{abstract}

\thanks{Research of the first and the second authors were partially supported  by NSERC}

\maketitle



\section{Introduction}
\label{introsect}

In the 1930-50s K.~Oka and H.~Cartan laid down the foundations of the modern function theory of several complex variables. In particular, they introduced the notion of a coherent sheaf and proved the following fundamental facts:

\vspace*{1mm}

(A) Every germ of a coherent sheaf $\mathcal A$ on a Stein manifold $X$ is generated by its global sections (``{\em Cartan theorem A}``).

(B) The sheaf cohomology groups $H^i(X,\mathcal A)$ ($i \geq 1$) are trivial (``{\em Cartan theorem B}``).

\vspace*{1mm}

Let us recall that a sheaf of modules over the sheaf of germs of holomorphic functions on $X$ is called \textit{coherent} if locally both the sheaf and its sheaf of relations are finitely generated. The class of coherent sheaves is closed under natural operations. 
Most sheaves that arise in complex analysis are coherent, see, e.g., \cite{GrRe} for details.

A {\em Stein manifold} is a complex manifold that admits a proper holomorphic embedding into some $\mathbb C^n$.

Cartan theorems A and B together with their numerous corollaries constitute the so-called Oka-Cartan theory of Stein manifolds. Applying these theorems one obtains solutions (in algebra $\mathcal O(X)$ of holomorphic functions on  a Stein manifold $X$) to all classical problems of function theory of several complex variables (such as Cousin problems, the Poincar\'{e} problem, the Levi problem, the problem of holomorphic extension from complex analytic subsets, corona problems and many others, see, e.g., \cite{GrRe}).

Further development of complex function theory was motivated, in part, by the problems requiring to study 
properties of holomorphic functions satisfying special conditions (e.g.,~certain growth conditions `at infinity').
As a result, the questions of whether the problems of the classical complex function theory can be solved within a proper subclass of $\mathcal O(X)$, e.g., consisting of holomorphic $L^p$-functions on $X$, $1 \leq p \leq \infty$, with respect to a suitable measure, started to play an important role (cf.~the pioneering papers of L.~H\"{o}rmander, J.J.~Kohn, C.~Morrey).
However, trying to incorporate in the proofs such conditions as $L^p$-summability, the classical Oka-Cartan theory encounters considerable difficulties. In particular, one has to amplify the sheaf-theoretic methods of Oka-Cartan, e.g., by integral representation formulas on complex manifolds, estimates for solutions of $\bar{\partial}$-equations, etc (see \cite{HL}).

Nevertheless, in some cases the methods of the Oka-Cartan theory can be extended to work within some special classes of holomorphic functions. The present paper studies one of these cases.

\begin{defin}
A holomorphic function $f$ defined on a regular covering $p:X \rightarrow X_0$ of a connected complex manifold $X_0$ with a deck transformation group $G$ is called a holomorphic $\mathfrak a$-function if 

(1) $f$ is bounded on subsets $p^{-1}(U_0)$, $U_0 \Subset X_0$, and

(2) for each $x \in X$ the function $G\ni g\mapsto f(g \cdot x)$ belongs to
a fixed closed unital 
subalgebra $\mathfrak a:=\mathfrak a(G)$ of the algebra $\ell_\infty(G)$ of bounded complex functions on $G$ (with pointwise multiplication and $\sup$-norm) that is invariant with respect to the action of $G$ on $\mathfrak a$ by right translations: 
$$
u \in \mathfrak a,~~g \in G \quad \Rightarrow \quad R_g u \in \mathfrak a,
$$
where $R_g(u)(h):=u(h g)$, $h\in G$.


We endow the subalgebra $\mathcal O_{\mathfrak a}(X) \subset \mathcal O(X)$ of holomorphic $\mathfrak a$-functions with the Fr\'{e}chet topology of uniform convergence on subsets $p^{-1}(U_0)$,  $U_0\Subset X_0$.

\end{defin}

The model examples of algebras $\mathfrak a$ and $\mathcal O_{\mathfrak a}(X)$ are given in Examples \ref{holap} and \ref{firstex} below.

In the present paper we obtain analogs of Cartan theorems A and B for coherent-type sheaves on the fibrewise compactification $c_{\mathfrak a}X$ of the covering $X$ of a Stein manifold $X_0$, a topological space having certain features of a complex manifold (see Definition \ref{compdefin} below). 
In our proofs we use some results and methods of the theory of coherent-type sheaves taking values in Banach or Fr\'{e}chet spaces, pioneered by E.~Bishop and L.~Bungart \cite{Bu,Bu2} and developed further by J.~Leiterer  \cite{Lt} (on Stein spaces), L.~Lempert \cite{Lemp} (on pseudoconvex subsets of Banach spaces with unconditional bases) and others. The constructions in \cite{Lemp} play particularly important role in our proofs.

In the second part of the work \cite{BK9}
we use our Cartan type theorems A and B to derive within subalgebra $\mathcal O_{\mathfrak a}(X)$ the basic results of complex function theory, including holomorphic extension from complex submanifolds, Cousin problems, properties of divisors, corona-type theorems, holomorphic Peter-Weyl-type approximation theorems, Hartogs-type theorems, describe uniqueness sets of holomorphic $\mathfrak a$-functions, etc.

\begin{example}[\textit{Holomorphic almost periodic functions}]
\label{holap}
The theory of almost periodic functions was created in the 1920s by H.~Bohr, who intended to apply it in the study of the study of the distribution of zeros of zeta-function in the critical strip. Nowadays almost periodic functions are used in in many areas of mathematics, including partial differential equations (e.g.,~KdV equation), harmonic analysis and number theory. A turning point in understanding of the nature of \textit{continuous almost periodic functions on} $\mathbb R$ came with the discovery of the so-called Bohr compactification of $\mathbb R$:
according to S.~Bochner, it quickly led to ``a sobering realization'' that the basic results of Bohr's theory on $\mathbb R$ can be deduced from Weyl's general theory of continuous functions on compact groups. In the same way, our work demonstrates that the basic results of the theory of \textit{holomorphic almost periodic functions} in tube domains follow from our Oka-Cartan type theory. The latter allows us to apply the modern methods of multidimensional complex function theory to holomorphic almost periodic functions, and obtain new results even in this classical setting.


Let us recall that a function $f \in \mathcal O(T)$ on a tube domain $T=\mathbb R^n+i\Omega \subset \mathbb C^n$, $\Omega \subset \mathbb R^n$ is open and convex, is called holomorphic almost periodic if the family of its translates $\{z \mapsto f(z+s),~z \in T\}_{s \in \mathbb R^n}$ is relatively compact in the topology of uniform convergence on tube subdomains $T'=\mathbb R^n+i\Omega'$, $\Omega' \Subset \Omega$. 
The cornerstone of Bohr's theory (see \cite{Bo}) is his approximation theorem stating that every holomorphic almost periodic function
is uniform limit (on tube subdomains $T'$ of $T$) of exponential polynomials
\begin{equation}
\label{expoly}
z\mapsto\sum_{k=1}^m c_ke^{i \langle z,\lambda_k\rangle}, \quad z\in T,\quad c_k \in \mathbb C, \quad \lambda_k \in \mathbb R^n,
\end{equation}
where $\langle\cdot,\cdot\rangle$ is the Hermitian inner product on $\mathbb C^n$.

The classical approach to study of holomorphic almost periodic functions exploits the fact that $T$ is the trivial bundle with base $\Omega$ and fibre $\mathbb R^n$ (e.g.,~as in the characterization of almost periodic functions in terms of their Jessen functions defined on $\Omega$, see, e.g.,~\cite{Sh,Lv,JT,Rn,FR,Torn}). 
By considering $T$ as a regular covering $p:T \rightarrow T_0\, (:=p(T) \subset \mathbb C^n)$ with the deck transformation group $\mathbb Z^n$, 
\begin{equation*}
p(z):=\bigl(e^{i z_1}, \dots, e^{i z_n}\bigr), \quad z=(z_1,\dots,z_n) \in T
\end{equation*}
(a complex strip covering an annulus if $n=1$), we obtain 

\begin{theorem}[\cite{BK9}]
\label{equivapthm} 
A function $f \in \mathcal O(T)$ is almost periodic
if and only if $f \in \mathcal O_{AP}(T)$.
\end{theorem}

Here $AP=AP(\mathbb Z^n)$ is the algebra of von Neumann's almost periodic functions on group $\mathbb Z^n$, i.e., those bounded complex functions whose families of translates are relatively compact in the topology of uniform convergence on $\mathbb Z^n$.

This result enables us to regard holomorphic almost periodic functions on $T$ as:

\begin{itemize}
\item[(a)] holomorphic sections of a certain holomorphic Banach vector bundle on $T_0$;

\item[(b)] holomorphic-type functions on the fibrewise Bohr compactification of the covering $p:T\rightarrow T_0$, a topological space having some properties of a complex manifold.
\end{itemize}



It is interesting to note that already in his monograph \cite{Bo} H.~Bohr uses equally often the aforementioned ``trivial fibre bundle'' and ``regular covering'' points of view on a complex strip. We note also that the Bohr compactification  of a tube domain $\mathbb R^n +i\Omega$ in the form $b\mathbb R^n +i\Omega$, where $b\mathbb R^n$ is the Bohr compactification of group $\mathbb R^n$, was used earlier in \cite{Fav, Fav2,Gri}.
\end{example}

\begin{example}
\label{firstex}
(1) Let $\mathfrak a:=\ell_\infty(G)$ be the algebra of all bounded complex functions on  the deck transformation group $G \cong p^{-1}(x)$, $x \in X_0$, of covering $p:X \rightarrow X_0$. 

By definition, every subalgebra $\mathcal O_{\mathfrak a}(X) \subset \mathcal O_{\ell_\infty}(X)$, $\ell_\infty:=\ell_\infty(G)$.

Algebra $\mathcal O_{\ell_\infty}(X)$ arises, e.g., in the study of holomorphic $L^2$-functions on coverings of pseudoconvex manifolds \cite{GHS,Br2, Br5,La}, Caratheodory hyperbolicity (Liouville property) of $X$ \cite{LS,Lin}, corona-type problems for bounded holomorphic functions on $X$ \cite{Br}. 
Earlier, some methods similar to those developed in the article were elaborated for algebra $\mathcal O_{\ell_\infty}(X)$ 
in \cite{Br}-\cite{Br4} in connection with corona-type problems for some subalgebras of bounded holomorphic functions on coverings of bordered Riemann surfaces, Hartogs-type theorems, integral representation of holomorphic functions of slow growth on coverings of Stein manifolds, etc; that work was motivated by the fact that if $X_0$ is compact, then $\mathcal O_{\ell_\infty}(X)=H^\infty(X)$, the algebra of all bounded holomorphic functions on $X$ (the most important cases are when $X$ is the unit ball or polydisk in $\mathbb C^n$).

A recent confirmation of potential productivity of the sheaf-theoretic approach to corona problem for $H^\infty$ comes from the recent paper \cite{BrBanach} on Banach-valued holomorphic functions on the unit disk $\mathbb D \subset \mathbb C$ having relatively compact images.

\smallskip

(2) Let $\mathfrak a:=c(G)$ (with ${\rm card~} G=\infty$) be the subalgebra of bounded complex functions on $G$ that admit continuous extensions to the one-point compactification of $G$ (here we consider $G$ equipped with discrete topology).
Then  $\mathcal O_{c}(X)$ consists of holomorphic functions having fibrewise limits at `infinity'.
\end{example}

\medskip

\noindent\textbf{Acknowledgement.~}We thank Professors T. Bloom, L. Lempert, T. Ohsawa, R. Shafikov and Y.-T. Siu for their interest to this work.

\tableofcontents

\section{Main results} 
\label{mainsect}

In some cases the maximal ideal space of algebra $\mathcal O_{\mathfrak a}(X)$ may be presented as a `fibrewise compactification' $c_{\mathfrak a}X$ of the covering $p:X \rightarrow X_0$. Now we briefly present this construction referring to Section \ref{moreoncomp} for further details. 

Let $M_{\mathfrak a}$ denote the maximal ideal space of algebra $\mathfrak a$, i.e., the space of all non-zero continuous complex homomorphisms of $\mathfrak a$ endowed with weak* topology (of $\mathfrak a^*$). 
The space $M_{\mathfrak a}$ is compact and Hausdorff, and
every element $f$ of $\mathfrak a$ determines a function $\hat{f} \in C(M_{\mathfrak a})$ by the formula
$$\hat{f}(\eta):=\eta(f), \quad \eta \in M_{\mathfrak a}.$$
Since algebra $\mathfrak a$ is uniform (i.e., $\|f^2\|=\|f\|^2$) and hence is semi-simple, the homomorphism $\hat{}: \mathfrak a \rightarrow C(M_{\mathfrak a})$ (called the {\em Gelfand transform}) is an isometric embedding (see, e.g., \cite{Gam}).
We have a continuous map $j=j_{\mathfrak a}:G \rightarrow M_{\mathfrak a}$ defined by 
associating to each point in $G$ its point evaluation homomorphism in $M_{\mathfrak a}$. 
This map is an injection if and only if algebra $\mathfrak a$ separates points of $G$.

Let $\hat{G}_{\mathfrak a}$ denote the closure of $j(G)$ in $M_{\mathfrak a}$. 
If algebra $\mathfrak a$ is self-adjoint (i.e., closed with respect to complex conjugation), then $\hat{}:\mathfrak a \rightarrow C(M_{\mathfrak a})$ is an isomorphism and hence $\hat{G}_{\mathfrak a}=M_{\mathfrak a}$.
The (right) action of group $G$ on itself by right multiplication induces the right action of $G$ on $M_{\mathfrak a}$ by the formula
$$\hat{R}_{g}(\eta)(f):=\eta(R_g(f)), \quad \eta \in M_{\mathfrak a}, \quad f \in \mathfrak a, \quad g \in G.$$
Then 
\begin{equation}
\label{jgact}
\hat{R}_g(j(h))=j(hg), \quad h,g \in G.
\end{equation}

The regular covering $p:X \rightarrow X_0$ can be viewed as a principal fibre bundle on $X_0$ with structure group $G$, that is there exists an open cover $(U_{0,\gamma})_{\gamma\in\Gamma}$ of $X_0$ and a locally constant cocycle $c=\{c_{\delta\gamma}:U_{0,\gamma} \cap U_{0,\delta} \rightarrow G\}_{\delta,\gamma\in\Gamma}$ so that the covering $p:X \rightarrow X_0$ can be obtained from the disjoint union 
$\sqcup_{\gamma}U_{0,\gamma} \times G$ by the identification
\begin{equation}
\label{extrem2}
U_{0,\delta} \times G \ni (x,g) \sim (x,g\cdot c_{\delta\gamma}(x)) \in U_{0,\gamma} \times G \quad \text{ for all }\quad x\in U_{0,\gamma} \cap U_{0,\delta},
\end{equation}
where projection $p$ is induced by the projections $U_{0,\gamma} \times G\rightarrow U_{0,\gamma}$ (see, e.g., \cite{Hirz}).

\begin{defin}
\label{compdefin}
The fibrewise ($\mathfrak a$-) compactification $\bar{p}:c_{\mathfrak a}X \rightarrow X_0$ is the fibre bundle on $X_0$ with fibre $\hat{G}_{\mathfrak a}$ associated to the principal bundle $p:X\rightarrow X_0$, i.e., $c_{\mathfrak a}X$ is obtained from the disjoint union $\sqcup_{\gamma}U_{0,\gamma} \times \hat{G}_{\mathfrak a}$ by the identification $$U_{0,\gamma} \times \hat{G}_{\mathfrak a} \ni (x,\omega) \sim (x,\hat{R}_{c_{\delta\gamma}(x)}(\omega) ) \in U_{0,\delta} \times \hat{G}_{\mathfrak a}, \quad \text{ for all }\quad x \in U_{0,\gamma} \cap U_{0,\delta},$$
where $\bar{p}$ is induced by projections $U_{0,\gamma}\times \hat{G}_{\mathfrak a}\rightarrow U_{0,\gamma}$.
\end{defin}

In \cite[Thm.~5.18]{BK9} we show that if subalgebra $\mathfrak a$ is self-adjoint and $X_0$ is a Stein manifold, then the maximal ideal space of algebra $\mathcal O_{\mathfrak a}(X)$ (i.e., the set of non-zero continuous complex homomorphisms of $\mathcal O_{\mathfrak a}(X)$ endowed with weak* topology of $\mathcal O_{\mathfrak a}(X)^*$) is homeomorphic to $c_{\mathfrak a}X$. (In particular, this is applied to algebras of Examples \ref{holap} and \ref{firstex}.)

Next, there exists a continuous map 
\begin{equation}
\label{iotamap}
\iota=\iota_{\mathfrak a}:X \rightarrow c_{\mathfrak a}X
\end{equation}
induced by the equivariant (with respect to the corresponding actions of $G$ on $G$ and $\hat{G}_{\mathfrak a}$) map $j$. Clearly, $\iota(X)$ is dense in $c_{\mathfrak a}X$ (thus, if $X_0$ is Stein and $\mathfrak a$ is self-adjoint, we have a corona-type theorem for algebra $\mathcal O_{\mathfrak a}(X)$ as its maximal ideal space is homeomorphic to $c_{\mathfrak a}X$, see \cite{BK9}). The map  $\iota$ is an injection if and only if $\mathfrak a$ separates points of $G$.

\begin{defin}\label{def2.2}
A function $f \in C(c_{\mathfrak a}X)$ is called {\em holomorphic} if its pullback $\iota^*f$ is holomorphic on $X$. 
The algebra of functions holomorphic on $c_{\mathfrak a}X$ is denoted by $\mathcal O(c_{\mathfrak a}X)$.
\end{defin}

\begin{proposition}
\label{basicpropthm}
The following is true:
\begin{itemize}

\item[(1)]
A function $f$ in $\mathcal O_{\mathfrak a}(X)$ determines a unique function $\hat{f}$ in $\mathcal O(c_{\mathfrak a}X)$ such that $\iota^*\hat{f}=f$.
Thus, there is a continuous embedding
$\mathcal O_{\mathfrak a}(X) \hookrightarrow \mathcal O(c_{\mathfrak a}X).$

\item[(2)]If $\mathfrak a$ is self-adjoint, then the correspondence $f\mapsto\hat{f}$ determines an isomorphism of algebras:
$\mathcal O_{\mathfrak a}(X) \cong \mathcal O(c_{\mathfrak a}X)$.
\end{itemize}
\end{proposition} 

So, for $\mathfrak a$ self-adjoint we can work with algebra $\mathcal O(c_{\mathfrak a}X)$ instead of $\mathcal O_{\mathfrak a}(X)$.


\begin{definition}
Let $U \subset c_{\mathfrak a}X$ be an open subset. 
A function $f \in C(U)$ is called {\em holomorphic} if $\iota^*f \in \mathcal O\bigl(\iota^{-1}(U)\bigr)$. 
The algebra of functions holomorphic on $U$ is denoted by $\mathcal O(U)$.
\end{definition}

Thus, we obtain the structure sheaf $\mathcal O:=\mathcal O_{c_{\mathfrak a}X}$ of germs of holomorphic functions on $c_{\mathfrak a}X$. 
Now, a {\em coherent sheaf} $\mathcal A$ on $c_{\mathfrak a}X$ is a sheaf of modules over $\mathcal O$ such that every point in $c_{\mathfrak a}X$ has an open neighbourhood $U$ over which, for any $N \geq 1$, there is a free resolution of $\mathcal A$ of length $N$, i.e., an exact sequence of sheaves of modules of the form
\begin{equation}
\label{coh0}
\mathcal O^{m_{N}}|_U \overset{\varphi_{N-1}}{\to} \dots \overset{\varphi_2}{\to} \mathcal O^{m_{2}}|_U \overset{\varphi_1}{\to} \mathcal O^{m_{1}}|_U \overset{\varphi_0}{\to} \mathcal A|_U \to 0,
\end{equation}
where $\varphi_i$, $0 \leq i \leq N-1$, are homomorphisms of sheaves of $\mathcal O$-modules.

If $X=X_0$ and $p=\Id$, then  this definition gives the classical definition of a coherent sheaf on a complex manifold. 

Our main results are stated as follows.

\medskip

Let $X_0$ be a Stein manifold, $\mathfrak a$ be self-adjoint, $\mathcal A$ be a coherent sheaf on $c_{\mathfrak a}X$.

\begin{theorem}[Cartan A]
\label{thmA}

Each stalk $\phantom{}_{x}\mathcal A$ ($x \in c_{\mathfrak a}X$) is generated as an $\phantom{}_{x}\mathcal O$-module by sections $\Gamma(c_{\mathfrak a}X,\mathcal A)$. 
\end{theorem}

\begin{theorem}[Cartan B]
\label{thmB}
Sheaf cohomology groups $H^i(c_{\mathfrak a}X,\mathcal A)=0$ for all $i \geq 1$.
\end{theorem}

The classical proof of Cartan theorems A and B on complex manifolds
does not work in our framework: there is no the Oka coherence lemma, the fibre $\hat{G}_{\mathfrak a}$ (in general, having infinite covering dimension -- e.g.~for almost periodic functions the fibre is an inverse limit of tori, cf.~Example \ref{compex}(2) below)  does not admit 
an open cover by contractible sets as needed for the proof of the classical Cartan lemma, etc.

Instead, we paste together the free resolutions \eqref{coh0} of a coherent sheaf $\mathcal A$ first over  sets $\bar{p}^{-1}(U_0)$, with $U_0 \subset X_0$ being open simply connected, using continuous partition of unity in $C(\hat{G}_{\mathfrak a})~(\cong \mathfrak a)$ and employing some constructions of \cite{Lemp}. Then we paste thus obtained free resolutions of $\mathcal A$ on sets $\bar{p}^{-1}(U_0)$ to obtain the free resolutions of $\mathcal A$ over preimages by $\bar{p}$ of relatively compact open subsets of $X_0$ that form an exhausting sequence. At this step we use the results of J.~Leiterer \cite{Lt} on Banach-valued coherent sheaves on Stein manifolds. Having obtained such global free resolutions of $\mathcal A$, we complete the proof of Theorems  \ref{thmA} and \ref{thmB} as in the classical case.

\begin{remark}

(1) An important example of a coherent sheaf on $c_{\mathfrak a}X$ is given by the sheaf of ideals of germs of holomorphic functions vanishing on a complex submanifold of $c_{\mathfrak a}X$. 
To define the latter, we will need the following notation. Let $U_0 \subset X_0$ be open and simply connected, $K \subset \hat{G}_{\mathfrak a}$ be open.
We denote by  $\bar{\psi}_{U_0}$ the trivialization $\bar{p}^{-1}(U_0) \rightarrow U_0 \times \hat{G}_{\mathfrak a}$ (see~Definition \ref{compdefin}), and define
$$\hat{\Pi}(U_0,K)=\bar{\psi}_{U_0}^{-1}(K) \subset c_{\mathfrak a}X.$$
In what follows, we identify $\hat{\Pi}(U_0,K)$ with $U_0 \times K$ (see subsection \ref{charts} for details).

\begin{definition}
\label{cmplsub}
A closed subset $Y \subset c_{\mathfrak a}X$ is called a complex submanifold of codimension $k$ if for every $x\in Y$ there exist a neighbourhood $U=\hat{\Pi}(U_0,K) \subset c_{\mathfrak a}X$ of $x$ and functions $h_1,\dots,h_k \in \mathcal O(U)$ such that

(1) $Y \cap U=\{x \in U: h_1(x)=\dots=h_k(x)=0\}$;

(2) The rank of map $z \mapsto \bigl(h_1(z,\omega),\dots,h_k(z,\omega)\bigr)$ is $k$ at each point $(z,\omega) \in Y \cap U$.
\end{definition}

 
We use coherence of the sheaf of ideals of $Y$ together with Theorem \ref{thmB} in \cite{BK9} to obtain results on interpolation within algebra $\mathcal O_{\mathfrak a}(X)$.  In the same paper we extend Cartan theorems A and B to work with coherent-type sheaves on complex submanifolds of $c_{\mathfrak a}X$.


(2) The assumption that subalgebra $\mathfrak a$ is self-adjoint is essential for our proofs of Theorems \ref{thmA} and \ref{thmB}. Nevertheless, using a technique different from that based on Theorem \ref{thmB}, we show in \cite{BK9} that the problem of interpolation within algebra $\mathcal O_{\mathfrak a}(X)$ can be solved (although for a restrictive class of sets) without the latter assumption. 


(3) An interesting example of the algebra of holomorphic $\mathfrak a$-functions with $\mathfrak a$ not self-adjoint is given by Bohr's holomorphic almost periodic functions on a horizontal strip domain $T \subset \mathbb C$ (see~Example \ref{holap}) whose restriction to each fibre $p^{-1}(x) \cong \mathbb Z$ belongs to the subalgebra $AP_+(\mathbb Z)$ of the von Neumann almost periodic functions on $\mathbb Z$ \textit{with positive spectra}, i.e.,
those functions on $\mathbb Z$ that admit uniform approximation by exponential polynomials $\sum_{k=1}^m c_k e^{i\lambda_k  t}$ ($t \in \mathbb Z$) with $\lambda_k \geq 0$ (see~\cite{Brcontr}). One can show that the maximal ideal space of algebra $\mathcal O_{AP_+}(T)$ can be presented as  inverse limit of an inverse limiting system of holomorphic fibre bundles over an annulus whose fibres are biholomorphic to disjoint unions of open polydisks, see~Example \ref{bohrex} below.

\end{remark}

\section{Examples}

\begin{example}[Examples of algebras $\mathfrak a$]
\label{exm}
In addition to algebras $\ell_\infty(G)$, $c(G)$, $AP(\mathbb Z^n)$ (see~Examples \ref{holap} and \ref{firstex}) we mention the following important examples of self-adjoint subalgebras of $\ell_\infty(G)$ invariant with respect to actions of $G$ by right translations.

(1) If a group $G$ is residually finite (respectively, residually nilpotent), i.e., for each element $g \in G$, $g \neq e$, there exists a normal subgroup $G_g \not\ni g$ such that $G/G_g$ is finite (respectively, nilpotent), we consider the closed subalgebra $\hat{\ell}_\infty(G) \subset \ell_{\infty}(G)$ generated by pullbacks to $G$ of algebras $\ell_\infty(G/G_g)$ for all $G_g$ as above.

(2) Recall that a bounded complex function $f$ on a (discrete) group $G$ is called almost periodic if the families of its left and right translates $$\{t \mapsto f(st)\}_{s \in G}, \quad \{t \mapsto f(ts)\}_{s \in G}$$ are relatively compact in $\ell_\infty(G)$ (J.~von Neumann \cite{N}). (It was proved in \cite{Ma}  that
the relative compactness of either the left of the right family of translates already gives almost periodicity.) 
The algebra of almost periodic functions on $G$ is denoted by $AP(G)$.

The basic examples of almost periodic functions on $G$ are matrix elements of finite-dimensional irreducible unitary representations of $G$.

Recall that a topological group $G$ is called \textit{maximally almost periodic} if its finite-dimensional irreducible unitary representations separate points of $G$.
Equivalently, $G$ is maximally almost periodic iff it admits a (continuous) monomorphism into a compact topological group. 

Any residually finite discrete group $G$
belongs to this class. In particular, $\mathbb Z^n$, finite groups, free groups, finitely generated nilpotent groups, pure braid groups, fundamental groups of three dimensional manifolds are maximally almost periodic. 

We denote by $AP_{\trig}(G) \subset AP(G)$ the space of functions
\begin{equation*}
t \mapsto \sum_{k=1}^m c_k \sigma^k_{ij}(t), \quad t\in G,\quad c_k \in \mathbb C, \quad \sigma^k=(\sigma^k_{ij}),
\end{equation*}
where $\sigma^k$, $1 \leq k \leq m$, are finite-dimensional irreducible unitary representations of $G$.
The von~Neumann approximation theorem states that {\em $AP_{\trig}(G)$ is dense in $AP(G)$} \cite{N}. 

In particular, algebra $AP(\mathbb Z^n)$ of almost periodic functions on $\mathbb Z^n$ contains as a dense subset the set of exponential polynomials
$t \mapsto \sum_{k=1}^m c_ke^{i\langle \lambda_k,t \rangle}$, $t \in \mathbb Z^n$, $c_k\in\mathbb C$, $\lambda_k \in \mathbb R^n.$
Here $\langle \cdot,\cdot \rangle$ is the standard inner product on $\mathbb R^n$.

(3) Algebra $AP_{\mathbb Q}(\mathbb Z^n)$ of almost periodic functions on $\mathbb Z^n$ with rational spectra. This is the subalgebra of $AP(\mathbb Z^n)$ generated (over $\mathbb C$) by functions of the form $t \mapsto e^{i\langle \lambda,t \rangle}$ with $\lambda \in \mathbb Q$.

(4) Suppose that the covering $p:X \rightarrow X_0$ is not regular. Still, we can define algebras $\mathcal O_{\ell_\infty}(X)$ and $\mathcal O_{c}(X)$ and include them in the framework of our theory by lifting them to the universal covering of $X_0$.
\end{example}

\begin{example}[Holomorphic almost periodic functions]
\label{holap3}

Elements of algebra \penalty-10000 $\mathcal O_{AP}(X)$, where $X \rightarrow X_0$ is a regular covering of a connected complex manifold $X_0$ with a deck transformation group $G$ and $AP:=AP(G)$, see Example \ref{exm} (2), are called \textit{holomorphic almost periodic functions}.
Equivalently, a function $f \in \mathcal O(X)$ is called holomorphic almost periodic if each orbit $\{g\cdot x\}_{g\in G}\subset X$ has an open neighbourhood $U\subset X$, invariant with respect to the action of $G$ on $X$, such that the family of translates $\{z \mapsto f(g \cdot z), z \in U\}_{g \in G}$ is relatively compact in the topology of uniform convergence on $U$ (see \cite{BK2} for the proof of the equivalence). 

This is a variant of the definition in \cite{W}, where $G$ is taken to be the group of all biholomorphic automorphisms of the complex manifold $X$. An interesting result in \cite{Ves} states that on Siegel domains of the second kind there are no non-constant holomorphic almost periodic functions in the sense of \cite{W} (although on Siegel domains of the first kind, i.e., on tube domains in $\mathbb C^n$, such holomorphic almost periodic functions even separate points). A similar result holds for algebra $\mathcal O_{AP}(X)$; for instance, if $X_0$ is a compact complex manifold, then all holomorphic almost periodic functions on $X$ are constant, see \cite[Theorem~2.3]{BK2}.

\end{example}

\begin{example}[Compactification of deck transformation group $G$]
\label{compex}
(1) Let $\mathfrak a:=c(G)$, ${\rm card}\,G=\infty$ (see~Example \ref{firstex}(2)). Then $\hat{G}_{c}$, $c:=c(G)$, is the one-point compactification of $G$.

(2) Let $\mathfrak a=AP(G)\, (=:AP)$ (see~Example \ref{exm}(2)). Then $\hat{G}_{AP}$ is homeomorphic to a compact topological group $bG$, called the {\em Bohr compactification} of $G$, uniquely determined by the universal property:
there exists a homomorphism $\mu:G \rightarrow bG$ such that for any compact topological group $H$ and any homomorphism $\nu: G\rightarrow H$ there exists a continuous homomorphism $\tilde\nu: bG\to H$ such that the following diagram 
$$
\bfig
\morphism(500,500)/{>}/<0,-500>[G`H;\nu] 
\morphism(500,500)/{>}/<500,0>[G`bG;\mu] 
\morphism(1000,500)|m|/{>}/<-500,-500>[bG`H;\ \, \tilde\nu]
\efig
$$
is commutative.


Applying this property to unitary groups $H:=U_n$, $n \geq 1$, we obtain that group $G$ is maximally almost periodic (see~Example \ref{exm}(2)) if and only if $\mu$ is a monomorphism.

The universal property implies that there exists a bijection between sets of finite-dimensional irreducible unitary representations of $G$ and $bG$. It turn, the Peter-Weyl theorem for $C(bG)$ and the von Neumann approximation theorem for $AP(G)$ (see~Example \ref{exm}(2)) imply that $AP(G) \cong C(bG)$. 
Therefore, $bG$ is homeomorphic to the maximal ideal space $M_{AP(G)}$ of algebra $AP(G)$ and $\mu(G)$ is dense in $bG$.
Under this homeomorphism, the set $j(G)$ (see Section \ref{mainsect} for its definition) is identified with the subgroup $\mu(G) \subset bG$. In case $G$ is maximally almost periodic, we will identify $G$ with $\mu(G)\subset bG$ by means of $\mu$ so that the action of $G$ on $\hat{G}_{AP}:=M_{AP(G)}$ coincides with that of $G$ on $bG$ by right translations.

By the Peter-Weil theorem the group $bG$ can be presented as inverse limit of an inverse system of finite-dimensional compact Lie groups. 
In particular, the Bohr compactification $b\mathbb Z$ of the group of integers $\mathbb Z$ is inverse limit of an inverse system of compact abelian Lie groups
$\mathbb T^k \times \oplus_{l=1}^m \mathbb Z/(n_l\mathbb Z)$, $k,m,n_l\in\mathbb N$, where
$\mathbb T^k:=(\mathbb S^1)^k$ is the real $k$-torus. It follows that $b\mathbb Z$ is disconnected and has infinite covering dimension.
The limit homomorphisms $b\mathbb Z \rightarrow \mathbb T^k \times \oplus_{l=1}^m \mathbb Z/(n_l\mathbb Z)$ are defined by finite families of characters $\mathbb Z\rightarrow\mathbb S^1$. 
For instance, let $\lambda_1$, $\lambda_2 \in \mathbb R \setminus \mathbb Q$ be linearly independent over $\mathbb Q$ and $\chi_{\lambda_i}:\mathbb Z\rightarrow\mathbb S^1$, $\chi_{\lambda_i}(n):=e^{2\pi i\lambda_i n}$, $i=1,2$, be the corresponding characters. Then the map $(\chi_{\lambda_{1}},\chi_{\lambda_2}):\mathbb Z\rightarrow\mathbb T^2$ is extended by continuity to a continuous surjective homomorphism $b\mathbb Z\rightarrow\mathbb T^2$. 
If $\lambda_1$, $\lambda_2$ are linearly dependent over $\mathbb Q$, then the corresponding extended homomorphism has image in $\mathbb T^2$ isomorphic to $\mathbb S^1 \times \mathbb Z/(m\mathbb Z)$ for some $m\in\mathbb N$.

(3) Let $\mathfrak a=AP_{\mathbb Q}(\mathbb Z^n)$ (see~Example \ref{exm}(3)).
Then $\hat{G}_{AP_{\mathbb Q}(\mathbb Z^n)}$ is homeomorphic to the profinite completion of group $\mathbb Z^n$; it is defined as inverse limit of an inverse system of groups
$\oplus_{l=1}^m \mathbb Z/(n_l\mathbb Z)$, $m,n_l\in\mathbb N$. 
It follows that covering dimension of $\hat{G}_{AP_{\mathbb Q}(\mathbb Z^n)}$ is zero.

(4) Let $\mathfrak a=\ell_\infty(G)\, (=:\ell_\infty)$ (see~Example \ref{firstex}(1)). Then $\hat{G}_{\ell_\infty} \cong \beta G$, the Stone--\v{C}ech compactification of group $G$. 
Covering dimension of $\hat{G}_{\ell_\infty}$ is zero (see, e.g., \cite[Thm.~9-5]{Nag}).
\end{example}

\section{Structure of fibrewise compactification $c_{\mathfrak a}X$}
\label{moreoncomp}

\subsection{Set structure of $c_{\mathfrak a}X$}
\label{structsection}

As a set, $c_{\mathfrak a}X$ is the disjoint union of connected complex manifolds, each is a covering of $X_0$.
Indeed, let $\Upsilon:=\hat{G}_{\mathfrak a}/G$ be the set of orbits of elements of $\hat{G}_{\mathfrak a}$ by the (right) action of $G$;
since any orbit $H \in \Upsilon$ is invariant with respect to the action of $G$, we may consider the associated to the action of $G$ on $H$ fibre bundle 
$p_H:X_H \rightarrow X_0$ with fibre $H$.
We assume that $H$ is endowed with discrete topology.
Then $p_H:X_H \rightarrow X_0$ is an unbranched covering of $X_0$ (in general non-regular).
Since $X$ is connected and $X$ is a covering of $X_H$, the complex manifold $X_H$ is connected as well. For each $H \in \Upsilon$ we have the natural continuous injective map $$\iota_H:X_H \hookrightarrow c_{\mathfrak a}X$$ determined by (equivariant with respect to the action of $G$) inclusion $H \hookrightarrow \hat{G}_{\mathfrak a}$. We denote $\hat{X}_H:=\iota_H(X_H)$.
In view of (\ref{jgact}), we have $j(G) \in \Upsilon$. Hence, if $j$ is injective, then $X=X_{j(G)}$ and $\iota=\iota_{j(G)}$ (see~(\ref{iotamap})).

It follows that as a set $$c_{\mathfrak a}X=\bigsqcup_{H \in \Upsilon}\iota_H(X_H).$$

\begin{example}
\label{c0ex}
$\hat{G}_{c}:=G\cup\{\infty\}$, ${\rm card}\, G=\infty$, is the one-point compactification of $G$,
and the action of $G$ on $\hat{G}_{c}$ fixes point $\infty$, so $\Upsilon=\{\{G\}, \{\infty\}\}$. It follows that the as a set $c_{c}X$ is the disjoint union of $X$ and $X_0$.
\end{example}

\begin{example}
\label{bohrex}
Let $\mathfrak a=AP(G)$ and $G$ be maximally almost periodic.
In what follows, we assume that $\hat{G}_{AP}$ is endowed with the group structure of the Bohr compactification $bG$, see~Example \ref{compex}(2).

Since $G$ is  a subgroup of $bG$, every orbit $H \in \Upsilon$ is a right coset of $G$ in $bG$,
 $X_H=X$ for all $H \in \Upsilon$, and each set $\hat{X}_H$ is dense in $c_{AP}X$.

The fibre bundle $c_{AP}X$ can be presented as inverse limit
of an inverse system of smooth fibre bundles on $X_0$.
Indeed, by the Peter-Weil theorem, $bG$ can be presented as inverse limit of an inverse system of finite-dimensional compact Lie groups $\{G_s\}_{s \in S}$ (see~Example \ref{compex}(2)). By $\pi_s: bG \rightarrow G_s$ we denote the corresponding limit homomorphisms. The right action of $G$ on $bG$ determines an action $r_s$ of $G$ on $G_s$, $r_s(g)(h):=h\cdot\pi_s(g)$, $g\in G$, $h\in G_s$. Let $p_s: X_s\rightarrow X_0$ be the associated to $r_s$ fibre bundle on $X_0$ with fibre $G_s$. Then $X_s$ has a smooth manifold structure and inverse limit along $\{G_s\}_{s \in S}$ determines inverse limit along the corresponding inverse system $\{X_s\}_{s\in S}$ which is homeomorphic to $c_{AP}X$.
\end{example}

\begin{example}
\label{ex3}
Let $\mathfrak a=AP_{\mathbb Q}(\mathbb Z^n)$.
Since covering dimension of $\hat{G}_{AP_{\mathbb Q}(\mathbb Z^n)}$ is zero (see~Example \ref{compex}(3)), covering dimension of $c_{AP_{\mathbb Q}(\mathbb Z^n)}X$ is equal to $\dim_{\mathbb R}X_0$.
\end{example}

\begin{example}
\label{ex1}
Let $\mathfrak a=\ell_\infty(G)$. Then $\hat{G}_{\ell_\infty} \cong \beta G$, the Stone--\v{C}ech compactification of group $G$. 
Since covering dimension of $\hat{G}_{\ell_\infty}$ is zero, covering dimension of $c_{\ell_\infty}X$ coincides with real dimension of $X_0$. 

It is easy to see that $c_{\ell_\infty}X$ is the maximal fibrewise compactification of covering $X\rightarrow X_0$ in the sense that if $\mathfrak a\subset\ell_\infty(G)$ is a (closed) subalgebra, then 
there exists a surjective bundle morphism $c_{\ell_\infty}X \rightarrow c_{\mathfrak a}X$. Indeed, let $\kappa:\hat{G}_{\ell_\infty} \rightarrow \hat{G}_{\mathfrak a}$ be a continuous surjective map adjoint to the inclusion $\mathfrak a \hookrightarrow \ell_\infty(G)$; since $\kappa$ is equivariant with respect to the corresponding (right) actions of $G$, it determines the required surjective bundle morphism.

\end{example}

\subsection{Complex structure on $c_{\mathfrak a}X$}
\label{charts}

A function $f \in C(U)$ on an open subset $U \subset c_{\mathfrak a}X$ is called {\em holomorphic} if 
$\iota^* f$ is holomorphic on $\iota^{-1}(U) \subset X$ in the usual sense. 

Let $U_0 \subset X_0$ be open. 
A function $f \in C(U)$ on an open subset $U\subset U_0 \times \hat{G}_{\mathfrak a}$ is called {\em holomorphic} if  $\tilde j^*f$, $\tilde j:={\rm Id}\times j :U_0\times G\rightarrow  U_0 \times \hat{G}_{\mathfrak a}$, is holomorphic on the open subset $\tilde j^{-1}(U)$ of the complex manifold $U_0\times G$ (see Section \ref{mainsect} for the definition of the map $j$).

For sets $U$ as above, let $\mathcal O(U)$ denote the algebra of holomorphic functions on $U$ endowed with the topology of uniform convergence on compact subsets of $U$.
Clearly, $f \in C(c_{\mathfrak a}X)$ belongs to $\mathcal O(c_{\mathfrak a}X)$ if and only if each point in $c_{\mathfrak a}X$ has an open neighbourhood $U$ such that $f|_{U} \in \mathcal O(U)$, see Definition \ref{def2.2}. 

By $\mathcal O_{U}$ we denote the sheaf of germs of holomorphic functions on $U$.

The category $\mathcal M$ of ringed spaces of the form $(U,\mathcal O_{U})$, where $U$ is either an open subset of $c_{\mathfrak a}X$ and $X$ is a regular covering of a complex manifold $X_0$ or is an open subset of $U_0 \times \hat{G}_{\mathfrak a}$ with $U_0\subset X_0$ open, contains in particular complex manifolds.

\begin{defin}
\label{defholmap}
A morphism of two objects in $\mathcal M$, that is, a map $F \in C(U_1,U_2)$, where $(U_i,\mathcal O_{U_i})\in\mathcal M$, $i=1,2$, such that $F^*\mathcal O_{U_2} \subset \mathcal O_{U_1}$, is called a {\em holomorphic map}.  
\end{defin}

The collection of holomorphic maps $F: U_1 \rightarrow U_2$, $(U_i,\mathcal O_{U_i})\in\mathcal M$, $i=1,2$, is denoted by $\mathcal O(U_1,U_2)$. If $F\in \mathcal O(U_1,U_2)$ has inverse $F^{-1}\in\mathcal O(U_2,U_1)$, then $F$ is called a {\em biholomorphism}.

The next result shows that, in a sense, the holomorphic structure on $c_{\mathfrak a}X$ is concentrated in 'horizontal layers' $\hat{X}_H \subset c_{\mathfrak a}X$ {\rm (}$H \in \Upsilon${\rm)}.


\begin{theorem}
\label{holmapthm}
For a connected complex manifold $M$ and a map  $F \in \mathcal O(M,c_{\mathfrak a}X)$ there exists $H \in \Upsilon$ such that $F(M) \subset \hat{X}_H$.
\end{theorem}

If covering dimension of $\hat{G}_{\mathfrak a}$ is zero (see~Examples \ref{ex3} and \ref{ex1}), then the assertion of Theorem \ref{holmapthm} holds true even for continuous maps, i.e., for any $F \in C(M,c_{\mathfrak a}X)$ there exists $H \in \Upsilon$ such that $F(M) \subset \hat{X}_H$ (see the argument in the proof of Theorem 1.2(d) in \cite{Br8}).

\medskip


Further, over each simply connected open subset $U_0 \subset X_0$ there exists a biholomorphic trivialization
$\psi=\psi_{U_0}:p^{-1}(U_0) \rightarrow U_0 \times G$ of covering $p:X \rightarrow X_0$ which is a morphism of fibre bundles with fibres $G$. 
Then there exists a biholomorphic trivialization $\bar{\psi}=\bar{\psi}_{U_0}:\bar{p}^{-1}(U_0) \rightarrow U_0 \times \hat{G}_{\mathfrak a}$
of bundle $c_{\mathfrak a}X$ over $U_0$, which is a morphism of fibre bundles with fibre $\hat{G}_{\mathfrak a}$, such that 
the following diagram 
\begin{equation*}
\bfig
\node a1(0,0)[p^{-1}(U_0)]
\node a2(0,-500)[U_0 \times G]
\node b1(700,0)[\bar{p}^{-1}(U_0)]
\node b2(700,-500)[U_0 \times \hat{G}_{\mathfrak a}]
\arrow[a1`a2;\psi]
\arrow[a1`b1;\iota]
\arrow[b1`b2;\bar{\psi}]
\arrow[a2`b2;\Id \times j]
\efig
\end{equation*}
is commutative.

For a given subset $S \subset G$ we denote
\begin{equation}
\label{pi1}
\Pi(U_0,S):=\psi^{-1}(U_0 \times S)
\end{equation}
and identify $\Pi(U_0,S)$ with $U_0 \times S$ where appropriate (here $\Pi(U_0,G)=p^{-1}(U_0)$).

For a subset $K \subset \hat{G}_{\mathfrak a}$ we denote
\begin{equation}
\label{pi2}
\hat{\Pi}(U_0,K)~\bigl(=\hat{\Pi}_{\mathfrak a}(U_0,K)\bigr):=\bar{\psi}^{-1}(U_0 \times K).
\end{equation}
A pair of the form $(\hat{\Pi}(U_0,K),\bar{\psi})$ will be called a \textit{coordinate chart} for $c_{\mathfrak a}X$. Similarly, sometimes we identify $\hat{\Pi}(U_0,K)$ with $U_0 \times K$.
If $K \subset \hat{G}_{\mathfrak a}$ is open, then, by our definitions, $\bar{\psi}^*: \mathcal O(U_0 \times K)\rightarrow\mathcal O(\hat{\Pi}(U_0,K))$ is an isomorphism of (topological) algebras.

\subsection{Basis of topology on $c_{\mathfrak a}X$}
\label{topsect}

We denote by $\mathfrak Q$ the basis of topology of $\hat{G}_{\mathfrak a}$ consisting of sets of the form
\begin{equation}
\label{base1}
\left\{\eta \in \hat{G}_{\mathfrak a}: \max_{1 \leq i \leq m}|h_i(\eta)-h_i(\eta_0)|<\varepsilon\right\}
\end{equation}
for $\eta_0 \in \hat{G}_{\mathfrak a}$, $h_1,\dots,h_m \in C(\hat{G}_{\mathfrak a})$, and $\varepsilon>0$.

The fibrewise compactification $c_{\mathfrak a}X$ is a paracompact Hausdorff space (as a fibre bundle with a paracompact base and a compact fibre); thus, $c_{\mathfrak a}X$ is a normal space.

It is easy to see that the family
\begin{equation}
\label{base2}
\mathfrak B:=\{\hat{\Pi}(V_0,L) \subset c_{\mathfrak a}X: V_0 \text{ is open simply connected in } X_0 \text{ and } L \in \mathfrak Q\}.
\end{equation}
forms a basis of topology of $c_{\mathfrak a}X$.

\subsection{Coherent sheaves on $c_{\mathfrak a}X$ }
\label{cohsect}

A sheaf of modules on an open subset $U \subset c_{\mathfrak a}X$ over $\mathcal O|_U$ will be called an \textit{analytic sheaf}. A homomorphism between analytic sheaves will be called an \textit{analytic homomorphism}.

Recall that a coherent sheaf $\mathcal A$ on $c_{\mathfrak a}X$ is an analytic sheaf such that every point in $c_{\mathfrak a}X$ has an open neighbourhood $U$ over which, for any $N \geq 1$, there is an exact sequence of sheaves of modules of the form
\begin{equation}\label{eq4.5}
\mathcal O^{m_{N}}|_U \overset{\varphi_{N-1}}{\to} \dots \overset{\varphi_2}{\to} \mathcal O^{m_{2}}|_U \overset{\varphi_1}{\to} \mathcal O^{m_{1}}|_U \overset{\varphi_0}{\to} \mathcal A|_U \to 0,
\end{equation}
where $\varphi_i$, $0 \leq i \leq N-1$, are analytic homomorphisms.

An analytic sheaf $\mathcal A$ on $c_{\mathfrak a}X$ is called a \textit{Fr\'{e}chet sheaf} if for each open set $U \in \mathfrak B$ the module of sections $\Gamma(U,\mathcal A)$ of $\mathcal A$ over $U$
is endowed with topology of a Fr\'{e}chet space.

\begin{proposition}
\label{frechetprop}
Every coherent sheaf can be turned in a unique way into a Fr\'{e}chet sheaf so that the following conditions are satisfied:

(1) If $\mathcal A$ is a coherent subsheaf of $\mathcal O$, then for any open subset $U \in \mathfrak B$ the module of sections $\Gamma(U,\mathcal A)$ has topology of uniform convergence on compact subsets of $U$.

(2) If $\mathcal A,\mathcal B$ are coherent sheaves on $c_{\mathfrak a}X$, then for any $U \in \mathfrak B$ the spaces $\Gamma(U,\mathcal A)$, $\Gamma(U,\mathcal B)$ are Fr\'{e}chet spaces, and any analytic homomorphism $\varphi:\mathcal A \rightarrow \mathcal B$ is continuous in the sense that the homomorphisms of sections of $\mathcal A$ and $\mathcal B$ over sets $U \in \mathfrak B$ induced by $\varphi$ are continuous.

The topology on $\Gamma(U,\mathcal A)$ can be defined by a family of semi-norms
\begin{equation*}
\|f\|_{V_k}:=\inf_h\left\{\sup_{x \in V_k}|h(x)|: h \in \Gamma(V_k,\mathcal O^{m_{1}}),~f=\bar{\varphi}_0(h)\right\},
\end{equation*}
where $\bar{\varphi}_0$ is the homomorphism of sections induced by $\varphi_0$ in \eqref{eq4.5}, and open sets $V_k \in \mathfrak B$ are such that $V_{k} \Subset V_{k+1} \Subset U$ for all $k$, and $U=\cup_k V_k$ (see Lemma \ref{exhlem}(2) below for existence of such exhaustion of $U$).

\end{proposition}

The proof essentially repeats that of an analogous result for coherent analytic sheaves on complex manifolds, see, e.g., \cite{GR}. For the sake of completeness, we provide the proof of the proposition in the Appendix.

%

\begin{theorem}[Runge-type approximation]
\label{cechtriv} Let $X_0$ be a Stein manifold, $\mathcal A$ a coherent sheaf on $c_{\mathfrak a}X$.
%
%
%
%
%
%
Suppose that $Y_0 \Subset X_0$, $\hat{Y} \subset c_{\mathfrak a}X$ are open and such that either
 (1) $Y_0$ is holomorphically convex in $X_0$ and $\hat{Y}=\bar{p}^{-1}(Y_0)$, or 
 (2) $Y_0$ is holomorphically convex in $X_0$ and is contained in a simply connected open subset of $X_0$, and $\hat{Y}=\hat{\Pi}(Y_0,K)$ for some $K \in \mathfrak Q$ (see subsection \ref{base1}). 
 
 Then the image of the restriction map $\Gamma(c_{\mathfrak a}X,\mathcal A) \rightarrow \Gamma(\hat{Y},\mathcal A)$ is dense (in the topology of Proposition \ref{frechetprop}).
\end{theorem}

\section{Proofs: preliminaries}

\SkipTocEntry\subsection{\v{C}ech cohomology}
\label{proofnotation}

For a topological space $X$ and a sheaf of abelian groups $\mathcal S$ on $X$ let $\Gamma(X,\mathcal S)$ denote the abelian group of continuous sections of $\mathcal S$ over $X$.

Let $\mathcal U$ be an open cover of $X$.
By $\mathcal C^i(\mathcal U,\mathcal S)$ we denote the space of \v{C}ech $i$-cochains with values in $\mathcal S$, by $\delta:\mathcal C^i(\mathcal U,\mathcal S) \rightarrow \mathcal C^{i+1}(\mathcal U,\mathcal R)$ the \v{C}ech coboundary operator (see, e.g., \cite{GrRe} for details), by
$\mathcal Z^i(\mathcal U,\mathcal S):=\{\sigma \in \mathcal  C^i(\mathcal U,\mathcal S): \delta \sigma=0\}$
the space of $i$-cocycles, and by $\mathcal B^i(\mathcal U,\mathcal S):=\{\sigma \in  \mathcal Z^i(\mathcal U,\mathcal S): \sigma=\delta( \eta), \eta \in \mathcal C^{i-1}(\mathcal U,\mathcal S)\}$
the space of $i$-coboundaries.
The \v{C}ech cohomology groups $H^i(\mathcal U,S)$, $i \geq 0$, are defined by $$H^i(\mathcal U,S):=\mathcal Z^i(\mathcal U,\mathcal S)/\mathcal B^i(\mathcal U,\mathcal S), \quad i \geq 1,$$ and $H^0(\mathcal U,\mathcal S):=\Gamma(\mathcal U,\mathcal S)$.

\SkipTocEntry\subsection{$\bar{\partial}$-equation}

Let $B$ be a complex Banach space, $D_0 \subset X_0$ be a strictly pseudoconvex domain in a complex manifold $X_0$. We fix a system of local coordinates on $D_0$ and consider a cover $\{W_{0,i}\}_{i \geq 1}$ of $D_0$ by coordinate patches.
By $\Lambda_b^{(0,q)}(D_0,B)$, $q \geq 0$, we denote the space of bounded continuous $B$-valued $(0,q)$-forms $\omega$ on $D_0$ endowed with norm
\begin{equation}
\label{approxsup}
\|\omega\|_{D_0}=\|\omega\|_{D_0}^{(0,q)}:=\sup_{x \in U_{i,0}, i \geq 1, \alpha}\|\omega_{\alpha,i} (x)\|_B, 
\end{equation}
where $\omega_{\alpha,i}$ ($\alpha$ is a multiindex) are coefficients of form $\omega|_{W_{0,i}} \in  \Lambda_b^{(0,q)}(W_{0,i},B)$ written in local coordinates on $W_{0,i}$. 

The next lemma follows easily from results in \cite{HL} (proved for $B=\mathbb C$), as all integral presentations and estimates are preserved when passing to the case of Banach-valued forms.

\begin{lemma}
\label{hl1}
There exists a bounded linear operator $$R_{D_0,B}\in \mathcal L\left(\Lambda_b^{(0,q)}(D_0,B), \Lambda_b^{(0,q-1)}(D_0,B)\right), \quad q \geq 1,$$ such that if $\omega \in \Lambda_b^{(0,q)}(D_0,B)$ is $C^\infty$ and satisfies $\bar{\partial}\omega=0$ on $D_0$, then $\bar{\partial} R_{D_0,B}\omega=\omega$ on $D_0$.
\end{lemma}

\section{Proof of Proposition \ref{basicpropthm}}

%
%
%
%

Given $f \in \mathcal O_{\mathfrak a}(X)$ denote $f_{x_0}:=f|_{p^{-1}(x_0)}$. Let $\hat{f}_{x_0} \in C(\hat{G}_{\mathfrak a})$ be such that $j^*\hat{f}_{x_0}=f_{x_0}$. The family $\{\hat{f}_{x_0}\}_{x_0\in X_0}$ determines a function $\hat f$ on $c_{\mathfrak a}X$ such that $\hat f(x)=\hat f_{x_0}(x)$ for $x_0:=\bar{p}(x)$. Using a normal family argument one shows that $\hat f\in\mathcal O(c_{\mathfrak a}X)$, see, e.g., \cite{Lin} or \cite[Lemma~2.3]{BrK} for similar results. Clearly, $\iota^*\hat f=f$.
Since the homomorphism $\,\hat{}:\mathfrak a \rightarrow C(\hat{G}_{\mathfrak a})$ is an injection, the constructed homomorphism $i:\mathcal O_{\mathfrak a}(X) \rightarrow \mathcal O(c_{\mathfrak a}X)$, $i(f):=\hat f$, is an injection as well. This completes the proof of the first assertion.

For the proof of the second assertion suppose that $\mathfrak a$ is self-adjoint. Then $\mathfrak a \cong C(\hat{G}_{\mathfrak a})$ and we can define the inverse homomorphism
$i^{-1}:\mathcal O(c_{\mathfrak a}X) \rightarrow \mathcal O(X)$ by the formula $$i(\hat{f}):=\iota^* \hat{f},\quad \hat{f} \in \mathcal O(c_{\mathfrak a}X).$$ 
Since $i^{-1}(\hat{f})|_{p^{-1}(x_0)}=j^*\bigl (\hat{f}|_{\bar{p}^{-1}(x_0)}\bigr) \in \mathfrak a$, $x_0 \in X_0,$ we have $i^{-1}(\hat{f}) \in \mathcal O_{\mathfrak a}(X)$, i.e., $i^{-1}$ maps $\mathcal O(c_{\mathfrak a}X)$ into $\mathcal O_{\mathfrak a}(X)$.

\section{Proofs of Theorems \ref{thmA}, \ref{thmB} and \ref{cechtriv}}

In what follows all polydisks are assumed to have finite polyradii.

\medskip

\begin{proof}[Proof of Theorem \ref{thmB}]

We will need the following results.

\begin{proposition}
\label{vanprop} 

Let $U:=\hat{\Pi}(U_0,K)$, where 
$U_0 \subset X_0$ is open and biholomorphic to an open polydisk in $\mathbb C^n$, and $K \in \mathfrak Q$ (see~(\ref{base1})). 

The following is true:

\begin{enumerate}
\item[(1)]
Let $\mathcal R$ be an analytic sheaf over $U$ having a free resolution of length $4N$
\begin{equation}
\label{frseq}
\mathcal O^{k_{4N}}|_U\overset{\varphi_{4N-1}}{\to} \dots \overset{\varphi_{2}}{\to} \mathcal O^{k_{2}}|_U\overset{\varphi_1}{\to} \mathcal O^{k_{1}}|_U \overset{\varphi_0}{\to} \mathcal R|_W \to 0.
\end{equation}
If $N \geq n:=\dim_{\mathbb C} U_0$, then the induced sequence of sections truncated to the $N$-th term
\begin{equation*}
\Gamma(U,\mathcal O^{k_{N}}) \overset{\bar{\varphi}_{N-1}}{\to} \dots \overset{\bar{\varphi}_{2}}{\to} \Gamma(U,\mathcal O^{k_{2}}) \overset{\bar{\varphi}_1}{\to} \Gamma(U,\mathcal O^{k_{1}}) \overset{\bar{\varphi}_0}{\to} \Gamma(U,\mathcal R) \to 0
\end{equation*}
is exact.

\medskip

\item[(2)] Suppose that free resolution (\ref{frseq}) exists for every $N$. Then $H^i(U,\mathcal R)=0$ for all $i \geq 1$.
\end{enumerate}
\end{proposition}

Let $\mathcal A$ be a coherent sheaf on $c_{\mathfrak a}X$. 

\begin{proposition}
\label{cohcor}
Every point $x_0 \in X_0$ has a neighbourhood $U_0$ such that for each $N \geq 1$ there exists a free resolution of sheaf $\mathcal A$ over $\bar{p}^{-1}(U_0)$ having length $N$  (see~Definition \ref{coh0}). 
\end{proposition}

(In other words, we may assume that the open sets $W$ in Definition \ref{coh0} have the form $U=\bar{p}^{-1}(U_0)$, $U_0 \subset X_0$ is open.)

We prove Propositions \ref{vanprop} and \ref{cohcor} in subsections \ref{vanproof} and \ref{cohproof}, respectively.

Now, let $\hat{\mathcal A}:=\bar{p}_*\mathcal A$ be the direct image of sheaf $\mathcal A$ under projection $\bar{p}:c_{\mathfrak a}X \rightarrow X_0$. By definition, $\hat{\mathcal A}$ is a sheaf of modules over the sheaf of rings $\mathcal O^{C(\hat{G}_{\mathfrak a})}$ of germs of holomorphic functions on $X_0$ taking values in the Banach space $C(\hat{G}_{\mathfrak a})$.
By Propositions  \ref{cohcor} and \ref{vanprop}(2) every $x_0 \in X_0$ has a basis of neighbourhoods $U_0$ such that $H^i(U,\mathcal A)=0$, $i \geq 1$, $U:=\bar{p}^{-1}(U_0)$. Therefore,
\begin{equation}
\label{directiso}
H^i(c_{\mathfrak a}X,\mathcal A) \cong H^i(X_0,\hat{\mathcal A}), \quad i \geq 0
\end{equation}
(see, e.g., \cite[Ch.~F, Cor.~6]{Gun3}).
We have $$\Gamma(U,\mathcal A) \cong \Gamma(U_0,\hat{\mathcal A}), \quad \Gamma(U,\mathcal O) \cong \Gamma(U_0,\mathcal O^{C(\hat{G}_{\mathfrak a})}).$$ It follows from  Proposition \ref{cohcor} and Proposition \ref{vanprop}(1) that for every $x_0 \in X_0$ and each $N \geq 1$ there exist a neighbourhood $U_0$ of $x_0$ and an exact sequence of sections
\begin{equation*}
\Gamma(U_0,(\mathcal O^{C(\hat{G}_{\mathfrak a})})^{k_N}) \to \dots \to \Gamma(U_0,(\mathcal O^{C(\hat{G}_{\mathfrak a})})^{k_1}) \to \Gamma(U_0,\hat{\mathcal A}) \to 0.
\end{equation*}
This implies that we have an exact sequence of sheaves
\begin{equation}
\label{exseq8}
(\mathcal O^{C(\hat{G}_{\mathfrak a})})^{k_{N}}|_{U_0} \to \dots \to (\mathcal O^{C(\hat{G}_{\mathfrak a})})^{k_{1}}|_{U_0}  \to \hat{\mathcal A} |_{U_0} \to 0.
\end{equation}

For every open set $U_0 \subset X_0$ the spaces of sections $\Gamma(U_0,\hat{\mathcal A})$, $\Gamma(U_0,\mathcal O^{C(\hat{G}_{\mathfrak a})})$ can be endowed with a Fr\'{e}chet topology so that the homomorphisms of sections induced by sheaf homomorphisms in (\ref{exseq8}) are continuous; indeed, since $\Gamma(U_0,\hat{\mathcal A}) \cong \Gamma(U,\mathcal A)$, $\Gamma(U_0,\mathcal O^{C(\hat{G}_{\mathfrak a})}) \cong \Gamma(U,\mathcal O)$, this follows from Proposition \ref{frechetprop} with $U=\bar{p}^{-1}(U_0)$.
Hence, in the terminology of \cite{Lt}, $\hat{\mathcal A}$ is a Banach coherent analytic Fr\'{e}chet sheaf. Therefore, according to Theorem 2.3 (iii) in \cite{Lt}, $H^i(X_0,\hat{\mathcal A})=0$ for all $i \geq 1$. Isomorphism (\ref{directiso}) now implies the required statement.
\end{proof}
\begin{proof}[Proof of Theorem \ref{cechtriv}]

Case (1). Due to the argument in the proof of Theorem \ref{thmB}, we have isomorphisms of Fr\'{e}chet spaces  $\Gamma(c_{\mathfrak a}X,\mathcal A) \cong \Gamma(X_0,\hat{\mathcal A})$, $\Gamma(\hat{Y},\mathcal A) \cong \Gamma(Y_0,\hat{\mathcal A})$. Now the result follows from Theorem 2.3 (iv) in \cite{Lt} applied to $\hat{\mathcal A}$.

Case (2). It suffices to show that the restriction map $\Gamma(\bar{p}^{-1}(Y_0),\mathcal A) \rightarrow \Gamma(\hat{Y},\mathcal A)$ has dense image and then to apply the result of case (1). 

We have $\hat{Y}=\hat{\Pi}(Y_0,K)$ for some  $Y_0 \Subset X_0$ open simply connected, and $K \in \mathfrak Q$. 
Since $\hat{Y} \in \mathfrak B$, we may use the last assertion of Proposition \ref{frechetprop}: it suffices to show that  given a section $f \in \Gamma(\hat{Y},\mathcal A)$ for every $\varepsilon>0$ and every $k$ there exists a section $\tilde{f}_k \in \Gamma(\bar{p}^{-1}(Y_0),\mathcal A)$ such that $\|f-\tilde{f}_k\|_{V_k}<\varepsilon$. 

Without loss of generality we may identify $\hat{Y}$ with $Y_0 \times K$, and $\bar{p}^{-1}(Y_0)$ with $Y_0 \times \hat{G}_{\mathfrak a}$ (see subsection \ref{charts}).
Then sets $V_k$ have the form
$V_k=V_{0,k} \times N_k$, where each $V_{0,k}$ is open and simply connected and $N_k \in \mathfrak Q$ are such that $N_{k} \Subset N_{k+1} \Subset K$ for all $k$, and $K=\cup_k N_k$ (see Lemma \ref{exhlem}(1) below).
Since space $\hat{G}_{\mathfrak a}$ is compact and, therefore, normal, for each $k$ there exists a function $\rho_k \in C(\hat{G}_{\mathfrak a})$ such that $0 \leq \rho_k \leq 1$ on $\hat{G}_{\mathfrak a}$, $\rho_k \equiv 1$ on $N_k$, and $\rho_k \equiv 0$ on $\hat{G}_{\mathfrak a} \setminus \bar{N}_{k+1}$. By definition, $\Gamma(Y_0 \times K,\mathcal A)$ is a module over $\Gamma(Y_0 \times K,\mathcal O)$, hence we can define $\tilde{f}_k:=\rho_k f \in \Gamma(Y_0 \times \hat{G}_{\mathfrak a},\mathcal A)$. Then $f-\tilde{f}_k=0$ on $Y_0 \times N_k$, so $\|f-\tilde{f}_k\|_{V_k}=0$. Thus, $\tilde{f}_k$ is the required approximation.
\end{proof}

\begin{proof}[Proof of Theorem \ref{thmA}]
Let $N \geq n$. Since sheaf $\mathcal A$ is coherent, there exists a neighbourhood $U$ of $x$ over which there is a free resolution 
\begin{equation}
\label{cohc}
\mathcal O^{m_{4N}}|_U\overset{\varphi_{4N-1}}{\to} \dots \overset{\varphi_{2}}{\to} \mathcal O^{m_{2}}|_U\overset{\varphi_1}{\to} \mathcal O^{m_{1}}|_U \overset{\varphi_0}{\to} \mathcal A|_U \to 0
\end{equation}
of length $4N$. It follows from the exactness of sequence (\ref{cohc}) that there exist sections $h_1,\dots,h_{m_1} \in \Gamma(U,\mathcal A)$ that generate $\phantom{}_{x}\mathcal A$ as an $\phantom{}_{x}\mathcal O$-module.
Now, it suffices to show that there exist a neighbourhood $V \subset U$ of $x$, global sections $f_1,\dots,f_{m_1} \in \Gamma(c_{\mathfrak a}X,\mathcal A)$ and functions $r_{ij} \in \mathcal O(V)$, $1 \leq i,j \leq m_1$, such that 
\begin{equation}
\label{repr1}
h_i|_{V}=\sum_{j=1}^{m_1} r_{ij}f_j|_{V}, \quad 1 \leq i \leq m_1.
\end{equation}
Without loss of generality we may assume that $U=\hat{\Pi}(U_0,K) \in \mathfrak B$, where $U_0 \subset X_0$ is biholomorphic to an open polydisk in $\mathbb C^n$ and is holomorphically convex in $X_0$, and $K \in \mathfrak Q$. By Proposition \ref{frechetprop} the topology on $\Gamma(W,\mathcal A)$ is determined by semi-norms
\begin{equation}
\label{seminorm}
\|h\|_{V_k}:=\inf_h\left\{\sup_{x \in V_k}|g(x)|: g \in \Gamma(V_k,\mathcal O^{m_{1}}),~h=\bar{\varphi}_0(g)\right\},
\end{equation}
where $\bar{\varphi}_0$ is the homomorphism of sections induced by $\varphi_0$ in (\ref{cohc}),
and open sets $V_k \in \mathfrak B$ are such that $V_{k} \Subset V_{k+1} \Subset W$ for all $k$, and $W=\cup_k V_k$, see~Lemma \ref{exhlem}(2) below; by definition, $V_k=V_{0,k} \times N_k$, where $V_{0,k} \Subset U_0$, $N_k \Subset K$ are open. Without loss of generality we may assume that each $V_{0,k}$ is biholomorphic to an open polydisk in $\mathbb C^n$ and is holomorphically convex in $X_0$.

Let $V:=V_{k_0}$, where $k_0$ is chosen so that $x \in V_{k_0}$.
It follows from the proof of Theorem \ref{cechtriv} (case (2) for $\hat{Y}:=U$)
that for every $\varepsilon>0$ there exist sections $f_1,\dots,f_{m_1} \in \Gamma(c_{\mathfrak a}X,\mathcal A)$ such that $\|h_i-f_i\|_{V}<\varepsilon$ for all $i$.
Now, by Proposition \ref{vanprop}(1) the sequence of sections corresponding to (\ref{cohc})
\begin{equation}
\label{coh09}
\dots \to \Gamma(V,\mathcal O^{m_{1}}) \overset{\bar{\varphi}_0}{\to} \Gamma(V,\mathcal A) \to 0
\end{equation}
is exact. Note that $\Gamma(V,\mathcal O^{m_1})$ consists of $m_1$-tuples of holomorphic functions on $V$. Let $\tilde{h}_i:=(0,\dots,1,\dots,0)$ ($1$ is in the $i$-th position), $1 \leq i \leq m_1$. Without loss of generality we may assume that
$h_i|_V=\bar{\varphi}_0(\tilde{h}_i)$. Since $\bar{\varphi}_0$ is surjective, there exist functions $\tilde{f}_i \in \Gamma(V,\mathcal O^{m_{1}})$ such that $\bar{\varphi}_0(\tilde{f}_i)=f_i|_{V}$. It follows from the definition of semi-norm $\|\cdot\|_V$, see~(\ref{seminorm}), that functions $\tilde{f}_i$ can be chosen in such a way that
\begin{equation}
\label{ineqtilde}
\sup_{x \in V}|\tilde{h}_i(x)-\tilde{f}_i(x)|<2\varepsilon.
\end{equation}
Since $\bar{\varphi}_0$ is a $\mathcal O(V)$-module homomorphism, the required identity (\ref{repr1}) would follow once we found functions $r_{ij} \in \Gamma(V,\mathcal O)$, $1 \leq i,j \leq m_1$, such that 
\begin{equation*}
\tilde{h}_i=\sum_{j=1}^{m_1} r_{ij}\tilde{f}_j, \quad 1 \leq i \leq m_1.
\end{equation*}
The latter system of linear equations (with respect to $r_{ij}$) can be rewritten as a matrix equation $H=FR$ with respect to $R=(r_{ij})_{i,j=1}^{m_1} \in \mathcal O\bigl(V,M_{n}(\mathbb C)\bigr)$, where $M_{n}(\mathbb C)$ denotes the set of $n \times n$ complex matrices, $H=(\tilde{h}_i)_{i=1}^{m_1} \in \mathcal O\bigl(V,GL_n(\mathbb C)\bigr)$ ($\tilde{h}_i$ are the columns of $H$) is the identity matrix, here $GL_n(\mathbb C) \subset M_{n}(\mathbb C)$ is the group of invertible matrices and $F=(\tilde{f}_i)_{i=1}^{m_1} \in \mathcal O\bigl(V,M_n(\mathbb C)\bigr)$ ($\tilde{f}_i$ are the columns of $F$). Since $\varepsilon>0$ can be chosen arbitrarily small, in view of (\ref{ineqtilde}) we may assume that $F \in \mathcal O\bigl(V,GL_n(\mathbb C)\bigr)$. Hence, we can define $R:=F^{-1}H$. 

This completes the proof of Theorem \ref{thmA}.
\end{proof}

\subsection{Auxiliary topological results}
For the proofs of Propositions \ref{vanprop} and \ref{cohcor} we will need the following results. 

Let $\mathcal L=\{L_i\}$ be an open cover of $\hat{G}_{\mathfrak a}$.
Recall that a \textit{refinement} of $\mathcal L$ is an open cover $\mathcal L'=\{L'_j\}$ of $\hat{G}_{\mathfrak a}$ such that each $L'_j \Subset L_{i}$ for some $i=i(j)$. 

Since $\hat{G}_{\mathfrak a}$ is compact, each open cover of $\hat{G}_{\mathfrak a}$ has a finite subcover.

\begin{lemma}
\label{coverlem}
Let $\mathcal L$ be a finite open cover of $\hat{G}_{\mathfrak a}$. There exist finite refinements $\mathcal L^k=\{L_j^k: L_j^k \in \mathfrak Q\}$ of $\mathcal L$  of the same cardinality such that $L_j^{k+1} \Subset L_j^k$ for all $j,k$.
\end{lemma}
\begin{proof}[Proof of Lemma \ref{coverlem}]
Since $\hat{G}_{\mathfrak a}$ is compact, there exists a finite refinement 
$\mathcal L'=\{L'_j\}$ of $\mathcal L=\{L_i\}$ such that every $L'_j \Subset L_{i}$ for some $i=i(j)$, and functions $\{\rho_j\} \subset C(\hat{G}_{\mathfrak a})$ such that $\rho_j \equiv 1$ on $\bar{L}'_j$, $\rho_j \equiv 0$ on $\hat{G}_{\mathfrak a} \setminus L_{i}$. We set $L_j^k:=\{\eta \in \hat{G}_{\mathfrak a}: \rho_j(\eta)>1-\frac{1}{2k}\}$, $k \geq 1$. By definition, $L_j^k \in \mathfrak Q$ for all $j$, $k$ (see~(\ref{base1})). It follows that $\mathcal L^k:=\{L_j^k\}$ are the required refinements of $\mathcal L$. 
\end{proof}


\begin{lemma}
\label{exhlem}
Let $K \in \mathfrak Q$, $U_0 \subset X_0$ be open. We set $U:=U_0 \times K$. 
The following is true:

\begin{enumerate}
\item[(1)] There exist open subsets $N_k \in \mathfrak Q$, $1 \leq k<\infty$, such that $N_{k} \Subset N_{k+1} \Subset K$ for all $k$ and $K=\cup_k N_k$.

\medskip

\item[(2)] 
There are open subsets $V_k=V_{0,k} \times N_k$, $1 \leq k<\infty$, such that $V_{k} \Subset V_{k+1} \Subset U$ for all $k$ and $U=\cup_k V_k$. Here $V_{0,k} \Subset U_0$ is open and $N_k \in \mathfrak Q$ for all $k$.

\medskip

\item[(3)] Let $L \in \mathfrak Q$ be such that $L \Subset K$. There exists a collection of sets $L^m \in \mathfrak Q$, $m \geq 1$, such that $L \Subset \dots \Subset L^{m+1} \Subset L^m \Subset \dots \Subset L^1 \Subset K$ for all $m$.

\medskip

\item[(4)] Let  $N \Subset K$ and $\{L_i\}$ be a finite collection of open subsets of $K$ such that $N \Subset \cup_i L_i$. There exists a finite number of open subsets $L_j' \subset K$, $L_j' \in \mathfrak Q$, such that $N \Subset \cup_j L_j'$ and for each $j$ we have $L_j' \Subset L_i$ for some $i=i(j)$.
\end{enumerate}

\end{lemma}
\begin{proof}
(1) Recall that the basis $\mathfrak Q$ of topology of $\hat{G}_{\mathfrak a}$ consists of sublevel sets of functions in $C(\hat{G}_{\mathfrak a})$, see~(\ref{base1}), so
$K=\{\eta \in \hat{G}_{\mathfrak a}: \max_{1 \leq i \leq m}|h_i(\eta))-h_i(\eta_0)|<\varepsilon\}$
for some $\eta_0 \in \hat{G}_{\mathfrak a}$, $ h_1,\dots,h_m \in C(\hat{G}_{\mathfrak a})$ and $\varepsilon>0$.
Let $\mathfrak a'$ be the subalgebra of $C(\hat{G}_{\mathfrak a})$ generated by functions $h_1,\dots,h_m,\bar{h}_1,\dots,\bar{h}_m$. Since algebra $\mathfrak a'$ is finitely generated, the maximal ideal space $M_{\mathfrak a'}$ of $\mathfrak a'$ is a compact subset of some $\mathbb C^p$, and we have $\mathfrak a' \cong C(M_{\mathfrak a'})$. The map $\pi:\hat{G}_{\mathfrak a} \rightarrow M_{\mathfrak a'}$ adjoint to the inclusion $\mathfrak a' \subset C(\hat{G}_{\mathfrak a})$ is proper and surjective. By definition, there exists an open subset $K' \subset M_{\mathfrak a'}$ such that $K=\pi^{-1}(K')$.
Since $M_{\mathfrak a'}$ is a compact metric space (as a compact subset of $\mathbb C^p$), there exist open subsets $N_k' \subset M_{\mathfrak a'}$ such that $N_{k-1}' \Subset N_k' \Subset K'$ for all $k$ and $K'=\cup_k N'_k$. We define $N_k:=\pi^{-1}(N_k) \in \mathfrak Q$. Clearly, each set $N_k'$ can be chosen in the form $N_k'=\left\{y \in M_{\mathfrak a'}: \max_{1 \leq i \leq r_k}|f_{ik}(y)-f_{ik}(y_0)|<\varepsilon\right\}$ for some $y_0 \in M_{\mathfrak a'}$, $f_{ik} \in C(M_{\mathfrak a'})$ and $\varepsilon>0$. Since $\pi^*C(M_{\mathfrak a'}) \subset C(\hat{G}_{\mathfrak a})$,  $N_k \in \mathfrak Q$ (see~(\ref{base1})) as required. 

A similar argument yields (3).

(2) It is clear that there exists a sequence of open sets $V_{0,k}$ such that $V_{0,k} \Subset V_{0,k+1} \Subset U_0$ for all $k$ and $U_0=\cup_k V_{0,k}$. 
We set $V_k:=V_{0,k} \times N_k$. 

(4) We apply Lemma \ref{coverlem} to the finite open cover of $\hat{G}_{\mathfrak a}$ consisting of the sets $L_i$ and set $\hat{G}_{\mathfrak a} \setminus \bar{N}$ to obtain a finite refinement $\{L'_j\} \subset \mathfrak Q$ of this cover. We exclude subsets $L'_j$ such that $L_j' \Subset \hat{G}_{\mathfrak a} \setminus \bar{N}$. Then for the obtained family $\bar{N} \subset \cup_j L_j'$ and by the definition of the refinement for each $j$ we have $L_j' \Subset L_i$ for some $i$, as required.
\end{proof}

\subsection{Proof of Proposition \ref{vanprop}}
\label{vanproof}
The (rather technical) proof of this proposition is presented at the end of this subsection. In the proof we will use the following preliminary results.

Let $U_0 \Subset \mathbb C^n$ be an open polydisk,
$K \in \mathfrak Q$ (see~(\ref{base1})). 
We set 
\begin{equation}
\label{u}
U:=U_0 \times K. 
\end{equation}
The sets $U$ and $\hat{\Pi}(U_0,K) \subset c_{\mathfrak a}X$ are biholomorphic (see~subsection \ref{charts}).
Definitions of a analytic homomorphism  and a free resolution (of an analytic sheaf over an open subset of $c_{\mathfrak a}X$, see~subsection \ref{cohsect}) are transferred naturally to analytic sheaves over $U$.
Thus, it suffices to prove Proposition \ref{vanprop} in the assumption that analytic sheaf $\mathcal R$ and free resolution (\ref{frseq}) are given over $U$.

A function $f \in C(U)$ is said to be $C^\infty$ if all its derivatives with respect to variable $x \in U_0$ (in some local coordinates on $U_0$) are in $C(U)$.
The algebra of $C^\infty$ functions on $U$ will be denoted by $C^\infty(U)$.

Let $\Lambda^{p,q}(U_0)$ be the collection of all $C^\infty$ $(p,q)$-forms on $U_0$. We define the space $\Lambda^{p,q}(U)$ of $C^\infty$ $(p,q)$-forms on $U$ by the formula
$\Lambda^{p,q}(U):=C^\infty(U) \otimes \Lambda^{p,q}(U_0)$. 
We have an operator $\bar{\partial}:\Lambda^{p,q}(U) \rightarrow \Lambda^{p,q+1}(U)$ defined as follows: 

Suppose that $\omega \in \Lambda^{p,q}(U)$ is given
 (in local coordinates on $U_0$) by the formula
\begin{equation*}
\omega=\sum_{|I|=p}\sum_{|J|=q}  f_{IJ} dz_I \wedge d\bar{z}_{J}, \quad f_{IJ} \in C^\infty(U), 
\end{equation*}
where $I=(i_1,\dots,i_p)$, $J=(j_1,\dots,j_q)$, $dz_I=dz_{i_1} \wedge \dots \wedge dz_{i_p}$, $d\bar{z}_J=d\bar{z}_{j_1} \wedge \dots \wedge d\bar{z}_{j_q}$;
then
\begin{equation}
\label{dbardef}
\bar{\partial} f:=\sum_{|I|=p}\sum_{|J|=q}  \bar{\partial} f_{IJ} \wedge dz_I \wedge d\bar{z}_J,
\end{equation} 
where 
\begin{equation*}
\bar{\partial} f_{IJ}(z,\eta):=\sum_{j=1}^n \frac{\partial f_{IJ}(z,\xi)}{\partial \bar{z}_j}d\bar{z}_j, \quad z=(z_1,\dots,z_n), \quad (z,\xi) \in U=U_0 \times K.
\end{equation*}

A form $\omega \in \Lambda^{p,q}(U)$ is called $\bar{\partial}$-closed if $\bar{\partial} \omega=0$.

Let $\Lambda^{p,q}$ be the sheaf of germs of $C^\infty$ $(p,q)$-forms on $U$, and $Z^{p,q} \subset \Lambda^{p,q}$ be the subsheaf of germs of $\bar{\partial}$-closed $(p,q)$-forms.
Note that $Z^{0,0}=\mathcal O$. 

\medskip

In what follows {\em we fix an open polydisk}
\begin{equation}
\label{v}
V_0 \Subset U_0.
\end{equation}
Let $W_0 \subset \bar{V}_0$ be open in $\bar{V}_0$ and such that $W_0=\bar{V}_0 \cap \tilde{W}_0$ for some product domain $\tilde{W}_0=\tilde{W}_{0}^1 \times \dots \times \tilde{W}_0^n \Subset U_0$, where each $\tilde{W}_0^i \Subset \mathbb C$ ($1 \leq i \leq n$) is simply connected and has smooth boundary (clearly, given any open neighbourhood of $\bar{W}_0$ in $U_0$, we can find such a set $\tilde{W}_0$  contained in this neighbourhood).

Fix a subset $W_0' \Subset W_0$ open in $\bar{V}_0$ and satisfying the same intersection condition as $W_0$.
Let 
\begin{equation}
\label{s}
S \subset K \text{ be a closed subset, and let } L' \Subset L \subset S \text{ be open in } S.
\end{equation}

\begin{lemma}
\label{dbarlem0}
For every $\omega \in \Gamma(W_0 \times L,Z^{0,q})$ there exists $\eta \in \Gamma(\bar{W}_0' \times \bar{L}',\Lambda^{0,q-1})$ such that $\bar{\partial} \eta=\omega$.
\end{lemma}

\begin{proof}
By definition, a section of sheaf $Z^{0,q}$ over $W_0 \times L$ is the restriction of a section of $Z^{0,q}$ over some open neighbourhood of $W_0 \times L$.
Therefore, we may assume that $L$ is open in $K$, and $\omega \in \Gamma(\tilde{W}_0 \times L,Z^{0,q})$ for some product domain $\tilde{W}_0$ as above. 

Clearly, there exists a product domain $\hat{W}_0 \Subset \tilde{W}_0$ open in $U_0$, where $\hat{W}_0=\hat{W}_{0}^1 \times \dots \times \hat{W}_0^n$ and each domain $\hat{W}_0^i \Subset \tilde{W}_0^i$ has smooth boundary, such that $W_0' \Subset \hat{W}_0$.
Further, since $\hat{G}_{\mathfrak a}$ is a normal space, there exists an open set $L'' \Subset L$ such that $L' \Subset L''$.

Let $C(\bar{L}'')$ be the Banach space of continuous functions on $\bar{L}''$ endowed with $\sup$-norm, $\Lambda^{0,q}(\tilde{W}_0,C(\bar{L}''))$ be the space of $C^\infty$ $C(\bar{L}'')$--valued $(0,q)$-forms on $\tilde{W}_0$, and
$$Z^{0,q}(\hat{W}_0,C(\bar{L}'')) \subset \Lambda^{0,q}(\hat{W}_0,C(\bar{L}''))$$ be the subspace of $\bar{\partial}_{C(\bar{L}'')}$-closed forms on $\hat{W}_0$. Here $$\bar{\partial}_{C(\bar{L}'')}:\Lambda^{0,q}(\hat{W}_0,C(\bar{L}'')) \rightarrow Z^{0,q+1}(\hat{W}_0,C(\bar{L}''))$$ is the usual operator of differentiation of $C(\bar{L}'')$-valued forms.

It is easy to see that the restriction to $\tilde{W}_0 \times \bar{L}''$ of a form in $\Gamma(\tilde{W}_0 \times L,\Lambda^{0,q})$ can be naturally identified with a form 
in $\Lambda^{0,q}(\tilde{W}_0,C(\bar{L}''))$ and, since $\hat{W}_0 \times L''$ is a neighbourhood of $\bar{W}_0' \times \bar{L}'$, every form in $\Lambda^{0,q}(\tilde{W}_0,C(\bar{L}''))$ determines (under such identification) a unique form in $\Gamma(\bar{W}_0' \times \bar{L}',\Lambda^{0,q})$; these identification maps commute with the actions of operators $\bar{\partial}$ and $\bar{\partial}_{C(\bar{L}'')}$.
In particular, form $\omega$ determines a form $\hat{\omega} \in Z^{0,q}(\tilde{W}_0,C(\bar{L}''))$.
Note that since $\tilde{W}_0 \Subset \mathbb C^n$ is a product domain, it is pseudoconvex. Hence $W_0$ admits an exhaustion by strictly pseudoconvex subdomains (see, e.g., \cite{Kra}).
Therefore, there exists a strictly pseudoconvex domain $D_0 \Subset \tilde{W}_0$ such that $\hat{W}_0 \Subset D_0$.
We restrict form $\hat{\omega}$ to $D_0$ (clearly, $\hat{\omega}|_{D_0}$ is bounded) and apply Lemma \ref{hl1}, where we take $B:=C(\bar{L}'')$.
We obtain that
there exists a form $\hat{\eta} \in \Lambda^{0,q-1}(\hat{W}_0,C(\bar{L}''))$ such that $\bar{\partial}_{C(\bar{L}'')} \hat{\eta}=\hat{\omega}$ over $\hat{W}_0$. It follows that the form $\eta \in \Gamma(\bar{W}_0' \times \bar{L}',\Lambda^{0,q-1})$ determined by 
$\hat{\eta}$ is the required one.
\end{proof}

%

\begin{defin}
\label{refdef}
We say that a finite open cover $\mathcal U=\{U_\alpha\}$ of $\bar{V}_0 \times S$ (see~(\ref{v}) and (\ref{s})) is \textit{of class} ($P$) if the following conditions are satisfied:

(1) $U_\alpha=U_{0,l} \times L_j$, $\alpha=(l,j)$, where $\{U_{0,l}\}$ and $\{L_j\}$ are finite open covers of $\bar{V}_0$ and $S$, respectively; 

(2) Each $L_j=S \cap \tilde{L}_j$ for some $\tilde{L}_j \in \mathfrak Q$ such that $\tilde{L}_j \subset K$;

(3) Each $U_{0,l}=\bar{V}_0 \cap \tilde{U}_{0,l}$ for some product domain $\tilde{U}_{0,l}=\tilde{U}_{0,l}^1 \times \dots \times \tilde{U}_{0,l}^n \Subset U_0$, where domains $\tilde{U}_{0,l}^i \Subset \mathbb C$ ($1 \leq i \leq n$) are simply connected and has smooth boundaries. 
\end{defin}

 \begin{lemma}
 \label{reflem}
 (1) Each open cover of $\bar{V}_0 \times S$ has a refinement of class ($P$).
 
 (2) Each open cover of $\bar{V}_0 \times S$ of class ($P$) has a refinement of class ($P$) of the same cardinality.
\end{lemma}
\begin{proof}
(1) There exists a refinement of a given open cover of $\bar{V}_0 \times S$ by open sets of the form $U_{0,l} \times O_j$, where $\{U_{0,l}\}$ and $\{O_i\}$ are finite open covers of $\bar{V}_0$ and $S$, respectively. By the definition of the induced topology on $S$, there exist open sets $\tilde{O}_i \subset K$ such that $O_i=S \cap \tilde{O}_i$. Now, we apply Lemma \ref{exhlem}(4) to $\{\tilde{O}_i\}$ (there we take $\bar{N}:=S$) to obtain open sets $\{\tilde{L}_j\} $ such that $L_j \Subset L_i$ for some $i=i(j)$ and $\tilde{L}_j \in \mathfrak Q$ for all $j$. Finally, we set $L_j:=S \cap \tilde{L}_j$. The sets $U_{0,l} \times L_j$ form the required refinement of class ($P$).

(2) Follows from  assertions (3) and (4) of Lemma \ref{exhlem}.
\end{proof}

Let $\mathcal U=\{U_\alpha:=U_{0,l} \times L_j\}$ be a finite open cover of $\bar{V}_0 \times S$ of class ($P$), and $\mathcal U'=\{U'_\alpha:=U'_{0,l} \times L'_j\}$ be a refinement of $\mathcal U$ of class ($P$) of the same cardinality (see~Lemma \ref{reflem}(2)). By definition, $\{U_{0,l}' \}$, $\{L_j'\}$ are refinements of open covers $\{U_{0,l}\}$ and $\{L_j\}$, respectively.

We have an injective refinement map $\iota_{\mathcal U,\mathcal U'}:\mathcal Z^i(\mathcal U,\mathcal R) \rightarrow \mathcal Z^i(\mathcal U',\mathcal R)$ (see subsection \ref{proofnotation} for notation). If no confusion arises,
we write $\sigma$ for $\iota_{\mathcal U,\mathcal U'}(\sigma)$.

\begin{lemma}
\label{cechlem0} The following is true:

\begin{enumerate}
\item[(1)] Let $\sigma \in \mathcal Z^i(\mathcal U,\mathcal O)$, $i \geq 1$. Then $\sigma \in \mathcal B^i(\mathcal U',\mathcal O)$.

\vspace*{1mm}

\item[(2)] $H^i(\bar{V}_0 \times S,\mathcal O)=0$, $i \geq 1.$
\end{enumerate}
\end{lemma}
\begin{proof}
(1) We will prove a more general result: if $\sigma \in \mathcal Z^i(\mathcal U, Z^{0,q})$, $i \geq 1$, $q \geq 0$, then  $\sigma \in \mathcal B^i(\mathcal U',Z^{0,q})$. In particular, taking $q=0$ we obtain assertion (1).

Let $i=1$,
$\sigma_1 \in \mathcal Z^1(\mathcal U,Z^{0,q})$. 
Since $\bar{V}_0 \times S$ is a paracompact space, there exist
partitions of unity $\{\lambda_l\}$ and $\{\rho_j\}$ subordinate to covers $\{U_{0,l}'\}$ and $\{L_j'\}$ ($C^\infty$ and continuous, respectively).
We define a $0$-cocycle $\sigma_0^\infty \in C^0(\mathcal U',\Lambda^{0,q})$ 
by the formula
\begin{equation}
\label{defFprime}
(\sigma^\infty_0)_\alpha(x,\xi):=\sum\limits_{\beta=(l,j)} \rho_j(\xi)\lambda_l(x) (\sigma_1)_{\beta,\alpha}(x,\xi), \quad (x,\xi) \in U'_\alpha \quad \text{ for all } \alpha.
\end{equation}
Since $(\sigma_1)_{\alpha,\beta}=(\delta \sigma_0^\infty)_{\alpha,\beta}=(\sigma_0^\infty)_\alpha-(\sigma_0^\infty)_\beta$ and $\bar{\partial} (\sigma_1)_{\alpha,\beta}=0$, the family $\{\bar{\partial} (\sigma_0^\infty)_\alpha\}$ determines $\omega\in\Gamma(\bar{V}_0 \times S,Z^{0,q+1})$,
$\omega|_{U_\alpha}:=\bar{\partial} (\sigma_0^\infty)_\alpha$. By Lemma \ref{dbarlem0} (with $W_0'=W_0=\bar{V}_0$ and $L'=L=S$) there exists $\eta \in \Gamma(\bar{V}_0 \times S,\Lambda^{0,q})$ such that $\bar{\partial}\eta=\omega$.
We
define a $0$-cochain $\sigma_0 \in \mathcal  C^0(\mathcal U',Z^{0,q})$ by the formula $(\sigma_0)_\alpha=(\sigma^\infty_0)_\alpha-\eta$. It follows that $\sigma_1=\delta \sigma_0$; therefore $\sigma_1 \in \mathcal B^1(\mathcal U',Z^{0,q})$.

Using Lemma \ref{reflem}(2) we may assume that there exists  a refinement $\mathcal U''=\{U''_\alpha:=U''_{0,l} \times L''_j\}$ of cover $\mathcal U$ of class ($P$) of the same cardinality as $\mathcal U$ such that $\mathcal U'$ is a refinement of $\mathcal U''$. 

Now, let $i>1$. Assume that we have shown for all $1 \leq l<i$, $q \geq 0$ that each $\sigma \in \mathcal Z^l(\mathcal U,Z^{0,q})$ belongs to $\mathcal B^l(\mathcal U'',Z^{0,q})$. For a given $\sigma_i \in \mathcal Z^i(\mathcal U,Z^{0,q})$
we define an $(i-1)$-cocycle $\sigma^\infty_{i-1} \in \mathcal C^{i-1}(\mathcal U'',\Lambda^{0,q})$ by the formula
\begin{equation*}
(\sigma^\infty_{i-1})_{\alpha_1,\dots,\alpha_i}(x,\xi):=\sum_{\beta=(l,j)} \rho_j(\xi)\lambda_l(x) (\sigma_i)_{\beta,\alpha_1,\dots,\alpha_i}(x,\xi), \quad (x,\xi) \in U''_{\alpha_1,\dots,\alpha_i}
\end{equation*}
for all $\alpha_1,\dots,\alpha_i$, where $U''_{\alpha_1,\dots,\alpha_i}:=\cap_{r=1}^i U''_{\alpha_r} \neq \varnothing$.

We have $\delta (\sigma_{i-1}^\infty)=\sigma_i$, so $\bar{\partial} \delta( \sigma_{i-1}^\infty)=\delta( \bar{\partial} \sigma_{i-1}^\infty)=0$. Define $\mu_{i-1}:=\bar{\partial} \sigma_{i-1}^\infty \in \mathcal C^{i-1}(\mathcal U'',Z^{0,q+1})$. Since $\delta( \mu_{i-1})=\bar{\partial} \mu_{i-1}=0$, by the induction assumption there exists an $(i-2)$-cochain $\mu_{i-2} \in \mathcal C^{i-2}(\mathcal U'',Z^{0,q})$ such that $\delta( \mu_{i-2})=\mu_{i-1}$ and $\bar{\partial} \mu_{i-2}=0$. 
Now, by Lemma \ref{dbarlem0}(1) there exists an $(i-2)$-cochain $\eta_{i-2} \in \mathcal C^{i-2}(\mathcal U',\Lambda^{0,q})$ such that $\bar{\partial} \eta_{i-2}=\mu_{i-2}$.
We define $\sigma_{i-1}:=\sigma_{i-1}^\infty-\delta(\eta_{i-2})$. Then $\delta( \sigma_{i-1})=\sigma_i$; so $\sigma_i \in \mathcal B^i(\mathcal U',Z^{0,q})$, as required. 

\medskip

(2) By Lemma \ref{reflem}(1) any open cover of $\bar{V}_0 \times S$ has a finite refinement of class ($P$), hence the required result follows from (1).
\end{proof}

Let $\{V_k\}_{k=1}^\infty$ be the exhaustion of $U$ by open sets obtained in Lemma \ref{exhlem}(2). By definition, each $V_k$ has the form $V_k=V_{0,k} \times N_k$, where $V_{0,k} \Subset U_0$, $N_k \Subset K$ are open, and $N_k \in \mathfrak Q$ for all $k$. Since $U_0$ is an open polydisk in $\mathbb C^n$, we may choose each $V_{0,k}$ to be an a open polydisk as well.

\begin{defin}[see~{\cite[Ch.IV]{GrRe}}]
\label{grdef}
We say that an analytic sheaf $\mathcal R$ on $U$ satisfies the \textit{Runge condition} if the following holds for every $k \geq 1$:

(a) The space of sections $\Gamma(\bar{V}_k,\mathcal R)$ is endowed with a semi-norm $|\cdot|_k$ such that $\Gamma(U,\mathcal R)|_{\bar{V}_k}$ is dense in  $\Gamma(\bar{V}_k,\mathcal R)$.

(b) There exist constants $M_k>0$ such that for every $f \in \Gamma(\bar{V}_{k+1},\mathcal R)$ we have $|f|_{\bar{V}_k}|_k \leq M_k |f|_{k+1}$.

(c) If $\{f_j\}$ is a Cauchy sequence in $\Gamma(\bar{V}_{k+1},\mathcal R)$, then  $\{f_j|_{\bar{V}_{k}}\}$ has a limit in $\Gamma(\bar{V}_{k},\mathcal R)$.

(d) If $f \in \Gamma(\bar{V}_{k+1},\mathcal R)$ and $|f|_{k+1}=0$, then $f|_{\bar{V}_{k}}=0$.
\end{defin}

\begin{lemma}[see~{\cite[Ch.IV]{GrRe}} for the proof]
\label{grrunge}Let $\mathcal R$ be an analytic sheaf on $U$. The following is true:
\begin{itemize} 
\item[(1)]
Suppose that $H^i(\bar{V}_k,\mathcal R)=0$ for all $i \geq 1$, $k \geq 1$. Then $H^i(U,\mathcal R)=0$ for all $i \geq 2$.

\item[(2)]If $\mathcal R$ satisfies the Runge condition and $H^1(\bar{V}_k,\mathcal R)=0$ for all $k \geq 1$, then $H^1(U,\mathcal R)=0$.
\end{itemize}
\end{lemma}

%

\begin{lemma}
\label{grlem}
The sheaf $\mathcal O|_U$ satisfies the Runge condition.
\end{lemma}
\begin{proof}
For a given section $f \in \Gamma(\bar{V}_{k},\mathcal O)$  let us denote by $\hat{f}(\omega) \in \mathbb C$ the value of germ $f(\omega)$ at point $\omega \in \bar{V}_{k}$.

We endow each space $\Gamma(\bar{V}_{k},\mathcal O)$ with semi-norm $|f|_k:=\sup_{\omega \in \bar{V}_{k}}|\hat{f}(\omega)|$. Conditions (b)--(d) are trivially satisfied. For the proof of (a), let us fix a section $f \in \Gamma(\bar{V}_{k},\mathcal O)$. 
By definition, a section of sheaf $\mathcal O$ over $\bar{V}_{k}:=\bar{V}_{0,k} \times \bar{N}_k$ is the restriction of a section of $\mathcal O$ over an open neighbourhood of $\bar{V}_{k}$. In particular, there exists an open neighbourhood $L \subset K$ of $\bar{N}_k$ such that section $f|_{\bar{V}_k}$ admits a bounded extension to $\bar{V}_{0,k} \times L$.
Since $\hat{G}_{\mathfrak a}~(\supset K)$ is a normal space, there exists a function $\rho_{k} \in C(K)$ such that $\rho_{k} \equiv 1$ on $\bar{N}_{k}$ and $\rho_{k} \equiv 0$ on $K \setminus L$. We set $\tilde{f}:=f \rho_{k} \in \Gamma(V_{0,k} \times K,\mathcal O)$. Then function $\tilde{f}$ determines a holomorphic function $\hat{f}$ defined in a neighbourhood of $\bar{V}_{0,k}$ with values 
in the Banach space $C_b(K)$ of bounded continuous functions on $K$ endowed with $\sup$-norm $\|\cdot\|$. We now apply the Runge-type approximation theorem for Banach-valued holomorphic functions, see \cite{Bu2}, to obtain that for every $\varepsilon>0$ there is a function $\hat{F} \in \mathcal O(U_0,C_b(K))$ such that $\sup_{x \in \bar{V}_{0,k}}\|\hat{f}(x)-\hat{F}(x)\|<\varepsilon.$ Then $\hat{F}$ determines a function $F \in \mathcal O(U)$ such that $\sup_{\omega \in \bar{V}_{k}}|f(\omega)-F(\omega)|<\varepsilon,$ which implies (a).
\end{proof}

\begin{corollary}
\label{cechlem1}
$H^i(U,\mathcal O)=0$ for all $i \geq 1.$
\end{corollary}
\begin{proof}
Follows from Lemmas \ref{cechlem0}(2), \ref{grrunge} and \ref{grlem}.
\end{proof}


\begin{lemma}
\label{surjlem}
Let $\mathcal B$, $\mathcal R$ be analytic sheaves on $U$.
Let $V_0 \Subset U_0$ be an open polydisk, $S \subset K$ a closed subset. 
Suppose that sequence
\begin{equation}
\label{cseq8}
\mathcal B \overset{\psi}{\to} \mathcal R \to 0
\end{equation}
is exact.
Then the sequence
\begin{equation}
\label{cseq9}
q_*(\mathcal B|_{\bar{V}_0 \times S}) \overset{q_*\psi}{\to} q_*(\mathcal R|_{\bar{V}_0 \times S}) \to 0
\end{equation}
is also exact. 
Here $q:\bar{V}_0 \times S \rightarrow \bar{V}_0$ is the projection on the first component and $q_*$ is the direct image functor.
\end{lemma}

\begin{proof}
We denote $\hat{\mathcal B}:=q_*(\mathcal B|_{\bar{V}_0 \times S})$, $\hat{\mathcal R}:=q_*(\mathcal R|_{\bar{V}_0 \times S})$, $\hat{\psi}:=q_*\psi$. We have to show that $\hat{\psi}$ is surjective. Given open subsets $W_0 \subset \bar{V}_0$, $L \subset S$ by 
$\Psi_{W_0 \times L}$ we denote the homomorphism of modules of sections $\Gamma(W_0 \times L,\mathcal B) \rightarrow \Gamma(W_0 \times L,\mathcal R)$ induced by $\psi$, and by $\hat{\Psi}_{W_0}$ the homomorphism of modules of sections $\Gamma(W_0 ,\hat{\mathcal B}) \rightarrow \Gamma(W_0,\hat{\mathcal R})$ induced by $\hat{\psi}$.
By the definition of direct image sheaf (see, e.g., \cite[Ch.~F]{Gun3})
\begin{equation}
\label{isov}
\Gamma(W_0 \times S,\mathcal B) \cong \Gamma(W_0,\hat{\mathcal B}), \quad \Gamma(W_0 \times S,\mathcal R) \cong \Gamma(W_0,\hat{\mathcal R}).
\end{equation}
To prove exactness of (\ref{cseq9}) it suffices to show that for every point $x_0 \in \bar{V}_0$, a neighbourhood $W_0 \subset \bar{V}_0$ of $x_0$, and a section $\hat{f}_{x_0} \in \Gamma(W_0,\hat{\mathcal R})$ there exists a section $\hat{g}_{x_0} \in \Gamma(\tilde{W}_0,\hat{\mathcal B})$ over a neighbourhood $\tilde{W}_0 \subset W_{0}$ of $x_0$ such that $\hat{\Psi}_{\tilde{W}_0}(\hat{g}_{x_0})=\hat{f}_{x_0}|_{\tilde{W}_0}$.

Let $f_{x_0} \in \Gamma(W_0 \times S,\mathcal R)$ be the section corresponding to $\hat{f}_{x_0}$ under the second isomorphism in (\ref{isov}). By definition, a section of sheaf $\mathcal R$ over $W_0 \times S$ is the restriction of a section of $\mathcal R$ over an open neighbourhood of $W_0 \times S$. Therefore, shrinking $W_0$, if necessary, we obtain that $f_{x_0}$ can be extended to a section of $\mathcal R$ over $W_0 \times M_1$, where $M_1 \subset K$ is an open neighbourhood of $S$. Since $\psi$ is a surjective sheaf homomorphism, for each point $y \in \{x_0\} \times M_1$ there exist open sets $W_{0,y} \subset W_0$, $L_y \subset M_1$  and a section $s_y \in \Gamma(W_{0,y} \times L_y,\mathcal B)$ such that $y \in W_{0,y} \times L_y$ and $\Psi_{W_{0,y} \times L_y}(s_y)=f_{x_0}|_{W_y \times L_y}$. Since space $\hat{G}_{\mathfrak a}~(\supset M_1)$ is compact  Hausdorff and, hence, is normal, there exists an open subset $M_2 \subset M_1$ such that $S \subset M_2$, and $\bar{M}_2 \subset M_1$. 
Since $\bar{M}_2$ is compact, there exist finitely many points $\{y_j\}_{j=1}^m \subset S$ such that $\bar{M}_2 \subset \cup_j L_{y_j}$. We set $\tilde{L}_{y_j}:=\bar{M}_2 \cap L_{y_j}$ for all $j$. There exists a partition of unity $\{\rho_j\} \subset C(\bar{M}_2)$ subordinate to $\{\tilde{L}_{y_j}\}$. 
We define $\tilde{W}_0:=\cap_j W_{0,y_j}$, and set
\begin{equation*}
g_{x_0}(z,\eta):=\sum_j \rho_j(\eta) s_{y_j}(z,\eta), \quad (z,\eta) \in \tilde{W}_0 \times S.
\end{equation*}
Then $g_{x_0} \in \Gamma(\tilde{W}_0 \times M_2,\mathcal B)$. We have
\begin{equation*}
\Psi_{\tilde{W}_0 \times S}(g_{x_0})=\sum_j \rho_j \Psi_{\tilde{W}_0 \times \tilde{L}_{y_j}}(s_{y_j})=\sum_{j} \rho_j f_{x_0}|_{\tilde{W}_0 \times \tilde{L}_{y_j}}=f_{x_0}|_{\tilde{W}_0 \times S}.
\end{equation*}
Let $\hat{g}_{x_0}$ denote the section in $\Gamma(\tilde{W}_{0},\hat{\mathcal B})$ corresponding to $g_{x_0}$ under the first isomorphism in (\ref{isov}). Then
$\hat{\Psi}_{\tilde{W}_{0}}(\hat{g}_{x_0})=\hat{f}_{x_0}|_{\tilde{W}_{0}},$
as required.
\end{proof}

\begin{defin}
\label{adef23}
We say that an analytic sheaf $\mathcal R$ (on $U$) \textit{admits a free resolution of length $N \geq 1$ over $U$} if there exists an exact sequence 
\begin{equation}
\label{aseq3}
\mathcal F_N|_{U} \overset{\varphi_{N-1}}{\to} \dots \overset{\varphi_{2}}{\to} \mathcal F_2|_{U} \overset{\varphi_1}{\to} \mathcal F_1|_{U} \overset{\varphi_0}{\to} \mathcal R \to 0,
\end{equation}
where $\mathcal F_i$ are free sheaves, i.e., sheaves of the form $\mathcal O^k$ for some $k \geq 0$ (by definition, $\mathcal O^0=\{0\}$).
\end{defin}

\begin{lemma}
\label{lem3comp}
Let $\mathcal R$ be an analytic sheaf on $U$ having a free resolution of length $3N$
\begin{equation}
\label{bseq1a}
\mathcal F_{3N}|_{U} \overset{\varphi_{3N-1}}{\to} \dots \overset{\varphi_{2}}{\to} \mathcal F_2|_{U} \overset{\varphi_1}{\to} \mathcal F_1|_{U} \overset{\varphi_0}{\to} \mathcal R \to 0.
\end{equation} 
If $N \geq n~(=\dim_{\mathbb C} U_0)$, then for each $k$ the induced sequence of sections
\begin{equation}
\label{seqsa}
\Gamma(\bar{V}_k,\mathcal F_{N}) \overset{\bar{\varphi}_{N-1}}{\to} \dots \overset{\bar{\varphi}_{2}}{\to} \Gamma(\bar{V}_k,\mathcal F_2) \overset{\bar{\varphi_1}}{\to} \Gamma(\bar{V}_k,\mathcal F_1) \overset{\bar{\varphi}_0}{\to} \Gamma(\bar{V}_k,\mathcal R) \to 0
\end{equation} 
is exact.
\end{lemma}

\begin{proof}
Let us fix $k \geq 1$.
Let
$q:\bar{V}_k \rightarrow \bar{V}_{0,k}$ be the projection, $q(x,\eta)=x$, $(x,\eta) \in V_k:=V_{0,k} \times N_k$ (see~notation before Definition \ref{grdef}). 
Let $q_*$ denote the direct image functor; set $\hat{\mathcal F}_i:=q_*(\mathcal F_i|_{\bar{V}_k})$, $\hat{\mathcal R}:=q_*(\mathcal R|_{\bar{V}_k})$, 
$\hat{\varphi}_i:=q_*\varphi_i$.
Applying $q_*$ to (\ref{bseq1a}) we obtain a complex of sheaf homomorphisms
\begin{equation}
\label{bseq3a}
\hat{\mathcal F}_{3N} \overset{\hat{\varphi}_{3N-1}}{\to} \dots \overset{\hat{\varphi}_{1}}{\to} \hat{\mathcal F}_1  \overset{\hat{\varphi}_{0}}{\to} \hat{\mathcal R} \to 0
\end{equation}
(a priori this sequence is not exact).
By the definition of a direct image sheaf, the sequence of sections of (\ref{bseq3a}) over $\bar{V}_{0,k}$ truncated to the $N$-th term
\begin{equation}
\label{bseq2a}
\Gamma(\bar{V}_{0,k},\hat{\mathcal F}_{N}) \to \dots \to \Gamma(\bar{V}_{0,k},\hat{\mathcal F_1}) \to \Gamma(\bar{V}_{0,k},\hat{\mathcal R}) \to 0
\end{equation}
coincides with sequence (\ref{seqsa}). 
Hence, the assertion would follow once we proved that sequence (\ref{bseq2a}) is exact.

Now, exact sequence (\ref{bseq1a}) yields the collection of short exact sequences
\begin{equation}
\label{bseq6a}
0 \to \mathcal R_i|_{\bar{V}_{k}} \overset{\iota}{\to} \mathcal F_i|_{\bar{V}_{k}}  \overset{\varphi_{i-1}}{\to} \mathcal R_{i-1}|_{\bar{V}_{k}} \to 0, \quad 1 \leq i \leq 3N-1,
\end{equation}
where $\mathcal R_i:=\Imag \varphi_{i}$ ($0 \leq i \leq 3N-1$), $\mathcal R_0:=\mathcal R$ and $\iota$ stands for inclusion. We apply to (\ref{bseq6a}) the direct image functor $q_*$ and Lemma \ref{surjlem}
to obtain the collection of short exact sequences (recall that $q_*$ is left exact, see, e.g., \cite[Ch.~F]{Gun3})
\begin{equation}
\label{bseqshorta}
0 \to \mathcal T_i \overset{\hat{\iota}}{\to} \hat{\mathcal F}_i \overset{\hat{\varphi}_{i-1}}{\to} \mathcal T_{i-1} \to 0, \quad 1 \leq i \leq 3N-1,
\end{equation}
where $\mathcal T_i:=q_*\mathcal R_i$ ($0 \leq i \leq 3N-1$).
An argument similar to that in the proof of Lemma \ref{cechlem0} implies $H^l(\bar{V}_{0,k},\hat{\mathcal F}_i)=0$, $l \geq 1$, $k \geq 1$, $1 \leq i \leq 3N$. Hence, each short exact sequence (\ref{bseqshorta}) yields a long exact sequence of the form
\begin{multline*}
0 \to \Gamma(\bar{V}_{0,k},\mathcal T_i) \to \Gamma(\bar{V}_{0,k},\hat{\mathcal F_i}) \to \Gamma(\bar{V}_{0,k},\mathcal T_{i-1}) \to \\
H^1(\bar{V}_{0,k},\mathcal T_i) \to 0 \to H^1(\bar{V}_{0,k},\mathcal T_{i-1}) \to H^2(\bar{V}_{0,k},\mathcal T_i) \to \dots\, .
\end{multline*}
Thus,
$H^m(\bar{V}_{0,k},\mathcal T_{i}) \cong H^{m+1}(\bar{V}_{0,k},\mathcal T_{i+1})$, $m \geq 1$, 
$1 \leq i \leq 3N-2$,
and so
\begin{equation*}
H^m(\bar{V}_{0,k},\mathcal T_{i}) \cong H^{m+l+1}(\bar{V}_{0,k},\mathcal T_{i+l+1}), \quad l \geq -1.
\end{equation*}

Let us take $m=1$, $1 \leq i \leq N$, $l:=2n-2$. Then
$$H^1(\bar{V}_{0,k},\mathcal T_{i}) \cong H^{2n+1}(\bar{V}_{0,k},\mathcal T_{i+2n-1}), \quad 1 \leq i \leq N.$$
Since $N \geq n$, we have $i+2n-1 \leq 3N-1$ for all $1 \leq i \leq N$; hence, $\mathcal T_{i+2n-1}$ is well defined for all $1 \leq i \leq N$. Since topological dimension of $\bar{V}_{0,k}$ is equal to $2n$, $H^{2n+1}(\bar{V}_{0,k},\mathcal T_{i+2n-1})=0$; therefore $H^1(\bar{V}_{0,k},\mathcal T_{i})=0$, $1 \leq i \leq N$.
Thus, we obtain the collection of short exact sequences
\begin{equation*}
0 \to \Gamma(\bar{V}_{0,k},\mathcal T_i) \to \Gamma(\bar{V}_{0,k},\hat{\mathcal F_i}) \to \Gamma(\bar{V}_{0,k},\mathcal T_{i-1}) \to 0, \quad 1 \leq i \leq N,
\end{equation*}
which yields exactness of sequence (\ref{bseq2a}). The proof is complete.
\end{proof}

\begin{lemma}
\label{runge2}
Let $\mathcal R$ be an analytic sheaf on $U$ having a free resolution of length $3N$
\begin{equation}
\label{seqb}
\mathcal F_{3N}|_{U} \overset{\varphi_{3N-1}}{\to} \dots \overset{\varphi_{2}}{\to} \mathcal F_2|_{U} \overset{\varphi_1}{\to} \mathcal F_1|_{U} \overset{\varphi_0}{\to} \mathcal R \to 0.
\end{equation}
If $N \geq n$, then $\mathcal R$ satisfies the Runge condition.
\end{lemma}
\begin{proof}
For a given section $h$ of sheaf $\mathcal O^m|_U$  by $\hat{h}(\omega) \in \mathbb C^m$ we denote the value of germ $h(\omega)$ at $\omega \in U$.
We have a short exact sequence
\begin{equation*}
0 \to \Ker \varphi_{0} \overset{\iota}{\to} \mathcal F_1|_U  \overset{\varphi_{0}}{\to} \mathcal R  \to 0,
\end{equation*}
where $\iota$ stands for inclusion.
In the proof of Lemma \ref{lem3comp} we have shown that, under the present assumptions, for each $k \geq 1$ the sequence of sections
\begin{equation*}
0 \to \Gamma(\bar{V}_k,\Ker \varphi_{0}) \overset{\bar{\iota}}{\to} \Gamma(\bar{V}_k,\mathcal F_1) \overset{\bar{\varphi}_{0}}{\to} \Gamma(\bar{V}_k,\mathcal R)  \to 0
\end{equation*}
is exact.
Given a section $h \in \Gamma(\bar{V}_k,\mathcal F_1)$, $\mathcal F_1:=\mathcal O^{m_1}$ for some $m_1\in\mathbb Z_+$, we define semi-norm $|h|_k:=\sup_{x \in \bar{V}_k} \|\hat{h}(x)\|$, where $\|\cdot\|$ is the Euclidean norm in $\mathbb C^{m_1}$.
Now, for a section $h \in \Gamma(\bar{V}_k,\mathcal R)$ we set
\begin{equation}
\label{fseminorm}
|f|_k:=\inf_h \{ |h|_k: h \in \Gamma(\bar{V}_k,\mathcal F_1), \bar{\varphi}_{0}(h)=f\}.
\end{equation}
We obtain a family of semi-norms $\{|\cdot|_k: k \geq 1\}$  on $\Gamma(U,\mathcal R)$. Let us show that for this family of semi-norms conditions (a)-(d) of Definition \ref{grdef} are satisfied.

(a) Let $f \in \Gamma(\bar{V}_{k},\mathcal R)$. There exists a section $h \in \Gamma(\bar{V}_{k},\mathcal F_1)$ such that $f=\bar{\varphi}_{0}(h)$. Using the same argument as in the proof of Lemma \ref{grlem}, we obtain that for any $\varepsilon>0$ there exists a section $\tilde{h} \in \Gamma(U,\mathcal F_1)$ such that $|\tilde{h}-h|_k<\varepsilon$. We set $\tilde{f}:=\bar{\varphi}_{0}(\tilde{h}) \in \Gamma(U,\mathcal R)$. By definition, $|\tilde{f}-f|_k<\varepsilon$, as required.

(b) Let $f \in \Gamma(\bar{V}_{k+1},\mathcal R)$. Since $$\{h \in \Gamma(\bar{V}_{k+1},\mathcal F_1), f=\bar{\varphi}_{0}(h)\}|_{\bar{V}_k} \subset  \{g \in \Gamma(\bar{V}_{k},\mathcal F_1), f|_{\bar{V}_k}=\bar{\varphi}_{0}(g)\}$$
and $|h|_k \leq |h|_{k+1}$ for every $h \in \Gamma(\bar{V}_{k+1},\mathcal F_1)$, condition (b) is satisfied with $M_k=1$ ($k \geq 1$) (see~(\ref{fseminorm})).

(c) Let $\{f_j\}$ be a Cauchy sequence in $\Gamma(\bar{V}_{k+1},\mathcal R)$. We must show that $\{f_j|_{\bar{V_k}}\}$ has a limit in the space $\bigl(\Gamma(\bar{V}_{k},\mathcal R), |\cdot |_k\bigr)$. In fact, there exists a Cauchy sequence $\{h_j\} \subset \Gamma(\bar{V}_{k+1},\mathcal O^{m_1})$ such that $f_j=\bar{\varphi}_{0}(h_j)$ for all $j$. Clearly, there exists a function $h \in \mathcal O(V_{k+1},\mathbb C^{m_1}) \cap C(\bar{V}_{k+1},\mathbb C^{m_1})$ such that 
\begin{equation}
\label{conv1}
\sup_{\omega \in \bar{V}_{k+1}} |h(\omega)-\hat{h}_j(\omega)| \rightarrow 0 \quad \text{ as } j \rightarrow \infty.
\end{equation}
Then $h \in \Gamma(\bar{V}_k,O^{m_1})$ and by (\ref{conv1}) $|h-h_j|_k \rightarrow 0$ as $j \rightarrow \infty$. Now, we set $f:=\bar{\varphi}_{0}(h) \in \Gamma(\bar{V}_k,\mathcal R)$, so by continuity $|f-f_j|_k \rightarrow 0$ as $j \rightarrow \infty$.

(d) Let $f \in \Gamma(\bar{V}_{k+1},\mathcal R)$, $|f|_{k+1}=0$. We must show that $f|_{\bar{V}_k}=0$. Indeed, by definition, there exists a sequence of sections $h_l \in \Gamma(\bar{V}_{k+1},\mathcal F_1)$ such that $f=\bar{\varphi}_{0}(h_l)$ for all $l$ and $\sup_{\omega \in \bar{V}_{k+1}}\|\hat{h}_l(\omega)\| \rightarrow 0$ as $l \rightarrow \infty$. Let $g_l:=h_1-h_l$, $l \geq 1$. Then $g_l \in \Gamma(\bar{V}_{k+1},\Ker \varphi_0)$ and 
\begin{equation}
\label{conv7}
\hat{g}_l(\omega)\rightarrow \hat{h}_1(\omega) \quad \text{uniformly in}\quad \omega \in \bar{V}_{k+1} \text{ as } l \rightarrow \infty. 
\end{equation}

Now, suppose to the contrary that $f|_{\bar{V}_k} \neq 0$. Then $h_1|_{\bar{V}_k} \not \in \Gamma(\bar{V}_k,\Ker \varphi_0)$.

Consider the second fragment of the free resolution of $\mathcal R$:
\begin{equation}
\label{sshort}
0 \to \Ker \varphi_{1} \overset{\iota}{\to} \mathcal F_2|_U  \overset{\varphi_{1}}{\to} \Ker \varphi_0  \to 0,
\end{equation}
and the corresponding sequence of sections (see~Lemma \ref{lem3comp})
\begin{equation}
\label{sectseq}
0 \to \Gamma(\bar{V}_{k+1},\Ker \varphi_{1}) \overset{\bar{\iota}}{\to} \Gamma(\bar{V}_{k+1},\mathcal F_2)  \overset{\bar{\varphi}_{1}}{\to}  \Gamma(\bar{V}_{k+1},\Ker\varphi_0)  \to 0,
\end{equation}
where $\bar{\varphi}_{1}$ is given by a matrix with entries in $\Gamma(\bar{V}_{k+1},\mathcal O)$.

Recall that $\Gamma(\bar{V}_{k+1},\mathcal F_2)$ is endowed with $\sup$-norm
\begin{equation}
\label{seminorm4}
|g|_{k+1}=\sup_{\omega \in \bar{V}_{k+1}}\|\hat{g}(\omega)\|, \quad g \in \Gamma(\bar{V}_{k+1},\mathcal F_2).
\end{equation}
Each section in the space $\Gamma(\bar{V}_{k+1},\mathcal F_2)$ determines a continuous function on $\bar{V}_{k+1}$ holomorphic in $V_{k+1}$.
Let $\mathcal A(\bar{V}_{k+1},\mathcal F_2)$ denote the completion of the space of these functions with respect to norm (\ref{seminorm4}). 
Analogously, we endow the space $\Gamma(\bar{V}_{k+1},\mathcal F_1)$ with $\sup$-norm and 
denote by $\mathcal A(\bar{V}_{k+1},\Ker\varphi_0)$ the completion (with respect to this norm) of its subspace $\Gamma(\bar{V}_{k+1},\Ker \varphi_0)$. 
Then (\ref{sectseq}) yields an exact sequence of Banach spaces
\begin{equation*}
\mathcal A(\bar{V}_{k+1},\mathcal F_2)  \overset{\bar{\varphi}_{1}}{\to} \mathcal A(\bar{V}_{k+1},\Ker\varphi_0)  \to 0.
\end{equation*}
It follows from (\ref{conv7}) that $\{g_l\}$, where $g_l$ are viewed as functions in $\mathcal A(\bar{V}_{k+1},\Ker \varphi_0)$,  is a Cauchy sequence and hence has a limit $g \in \mathcal A(\bar{V}_{k+1},\Ker \varphi_0)$. Then there exists $r \in \mathcal A(\bar{V}_{k+1},\mathcal F_2)$ such that $g=\bar{\varphi}_{1}(r)$. 
Also, by Lemma \ref{lem3comp} we obtain that sequence (\ref{sshort}) induces an exact sequence of sections
$$
0 \to \Gamma(\bar{V}_{k},\Ker \varphi_{1})  \overset{\bar{\iota}|_{\bar{V}_k}}{\to} \Gamma(\bar{V}_{k},\mathcal F_2)  \overset{\bar{\varphi}_{1}|_{\bar{V}_{k}}}{\to}  \Gamma(\bar{V}_{k},\Ker\varphi_0)  \to 0.
$$
Clearly, we have
$$
A(\bar{V}_{k+1},\mathcal F_2)|_{\bar{V}_k}\subset \Gamma(\bar{V}_k,\mathcal F_2) \text{ and } \mathcal A(\bar{V}_{k+1},\Ker\varphi_0)|_{\bar{V}_k}\subset \Gamma(\bar{V}_k,\Ker \varphi_0).
$$
Hence, $r|_{\bar{V}_k} \in \Gamma(\bar{V}_k,\mathcal F_2)$ and
$g|_{\bar{V}_k}=\bar{\varphi}_1(r|_{\bar{V}_k}) \in \Gamma(\bar{V}_k,\Ker \varphi_0)$.
Since, by our definition, $h_1|_{\bar{V}_k}=g|_{\bar{V}_k}$, we have $h_1|_{\bar{V_k}}\in \Gamma(\bar{V}_k,\Ker \varphi_0)$, which contradicts the assumption $f|_{\bar{V}_k} \neq 0$.
\end{proof}

\begin{lemma}
\label{sectionslem}
Let $\mathcal R$ be an analytic sheaf over $U$ admitting a free resolution of length $4N$
\begin{equation}
\label{sects1}
\mathcal F_{4N}|_{U} \overset{\varphi_{4N-1}}{\to} \dots \overset{\varphi_{2}}{\to} \mathcal F_2|_{U} \overset{\varphi_1}{\to} \mathcal F_1|_{U} \overset{\varphi_0}{\to} \mathcal R \to 0.
\end{equation} 
If $N \geq n$, and for each $k$ the sequence of sections
\begin{equation}
\label{sects2}
\Gamma(\bar{V}_k,\mathcal F_{N}) \overset{\bar{\varphi}_{N-1}}{\to} \dots \overset{\bar{\varphi}_{2}}{\to} \Gamma(\bar{V}_k,\mathcal F_2) \overset{\bar{\varphi}_1}{\to} \Gamma(\bar{V}_k,\mathcal F_1) \overset{\bar{\varphi}_0}{\to} \Gamma(\bar{V}_k,\mathcal R) \to 0
\end{equation} 
is exact,
then the sequence of sections
\begin{equation}
\label{sects3}
\Gamma(U,\mathcal F_{N}) \overset{\bar{\varphi}_{N-1}}{\to} \dots \overset{\bar{\varphi}_{2}}{\to} \Gamma(U,\mathcal F_2) \overset{\bar{\varphi}_1}{\to} \Gamma(U,\mathcal F_1) \overset{\bar{\varphi}_0}{\to} \Gamma(U,\mathcal R) \to 0
\end{equation}
is also exact. 
\end{lemma}
\begin{proof}
Exact sequence (\ref{sects1}) yields the collection of short exact sequences
\begin{equation}
\label{sects5}
0 \to \mathcal R_i \overset{\iota}{\to} \mathcal F_i|_U  \overset{\varphi_{i-1}}{\to} \mathcal R_{i-1} \to 0, \quad 1 \leq i \leq N-1,
\end{equation}
where $\mathcal R_i:=\Imag \varphi_{i}$ ($0 \leq i \leq N-1$), $\mathcal R_0:=\mathcal R$, and $\iota$ stands for inclusion. Recall that the section functor $\Gamma$ is left exact (see, e.g., \cite[Ch.\,3]{Gun3}), hence we have the collection of exact sequences
\begin{equation*}
0 \to \Gamma(U,\mathcal R_i) \overset{\bar{\iota}}{\to} \Gamma(U,\mathcal F_i) \overset{\bar{\varphi}_{i-1}}{\to} \Gamma(U,\mathcal R_{i-1}), \quad 1 \leq i \leq N-1.
\end{equation*}
It suffices to show that $\bar{\varphi}_{i-1}$ is surjective; this would imply that (\ref{sects3}) is exact.

It follows from the exactness of sequence (\ref{sects2}) that for each $k$ the sequences
\begin{equation}
\label{sects4}
0 \to \Gamma(\bar{V}_k,\mathcal R_i) \overset{\bar{\iota}}{\to} \Gamma(\bar{V}_k,\mathcal F_i) \overset{\bar{\varphi}_{i-1}}{\to} \Gamma(\bar{V}_k,\mathcal R_{i-1}) \to 0, \quad 1 \leq i \leq N-1,
\end{equation}
are exact. By Lemma \ref{cechlem0} $H^1(\bar{V}_k,\mathcal F_i)=0$, $1 \leq i \leq N$, for all $k \geq 1$, therefore the long exact sequence induced by (\ref{sects5}) over $\bar{V}_k$ has the form
\begin{multline}
\notag
0 \to \Gamma(\bar{V}_k,\mathcal R_i) \to \Gamma(\bar{V}_k,\mathcal F_i) \to \Gamma(\bar{V}_k,\mathcal R_{i-1}) \to \\ H^1(\bar{V}_k,\mathcal R_i) \to 0 \to H^1(\bar{V}_k,\mathcal R_{i-1}) \to H^2(\bar{V}_k,\mathcal R_i) \to \dots, \quad 1 \leq i \leq N-1.
\end{multline}
Now it follows from (\ref{sects4}) that $H^1(\bar{V}_k,\mathcal R_i)=0$ for all $k \geq 1$, $1 \leq i \leq N-1$. 

The long exact sequence induced by (\ref{sects5}) over $U$ has the form
\begin{multline}
\label{sects7}
0 \to \Gamma(U,\mathcal R_i) \to \Gamma(U,\mathcal F_i) \to \Gamma(U,\mathcal R_{i-1}) \to \\ H^1(U,\mathcal R_i) \to H^1(U,\mathcal F_i) \to H^1(U,\mathcal R_{i-1}) \to H^2(U,\mathcal R_i) \to \dots, \quad 1 \leq i \leq N-1.
\end{multline}
Each sheaf $\mathcal R_i$, $1 \leq i \leq N-1$, has a free resolution of length $3N$, hence by Lemma \ref{runge2} it satisfies the Runge condition.
It follows from Lemma \ref{grrunge} (2) that $H^1(U,\mathcal R_i)=0$ for all $1 \leq i \leq N-1$. We obtain from (\ref{sects7}) that sequences 
\begin{equation*}
0 \to \Gamma(U,\mathcal R_i) \overset{\bar{\iota}}{\to} \Gamma(U,\mathcal F_i) \overset{\bar{\varphi}_{i-1}}{\to} \Gamma(U,\mathcal R_{i-1}) \to 0, \quad 1 \leq i \leq N-1,
\end{equation*}
are exact, which implies exactness of sequence (\ref{sects3}).
\end{proof}

\begin{proof}[Proof of Proposition \ref{vanprop}]

(1) Follows from Lemmas \ref{lem3comp} and \ref{sectionslem}.

(2) According to Lemma \ref{runge2} sheaf $\mathcal R$ satisfies the Runge condition. Hence, by Lemma \ref{grrunge} we only have to show that $H^i(\bar{V}_k,\mathcal R)=0$ for all $i \geq 1$ and $k \geq 1$.

Let $\mathcal V$ be an open cover of $\bar{V}_k:=\bar{V}_{0,k} \times \bar{K}$. 
It suffices to show that given an $i$-cocycle $\sigma \in \mathcal Z^i(\mathcal V,\mathcal R)$ (see~notation before Lemma \ref{cechlem0})  there exists a refinement $\mathcal V'$ of $\mathcal V$ such that 
the image of $\sigma$ by the refinement map 
$\mathcal Z^i(\mathcal V,\mathcal R) \rightarrow \mathcal Z^i(\mathcal V',\mathcal R)$
belongs to $\mathcal B^i(\mathcal V',\mathcal R)$. 

By Lemma \ref{reflem}\,(1) there exists a finite refinement $\mathcal U=\{U_\alpha\}$, $U_\alpha:=U_{0,l} \times L_j$, $\alpha=(l,j)$, of cover $\mathcal V$ of class ($P$) (see~Definition \ref{refdef}).
Let $s=s_{\mathcal U}$ be the number of elements of $\mathcal U$ and let $N \geq \max\{n,s\}$ be the length of the free resolution of $\mathcal R$ over $U$.
By the definition of an open cover of class $(P)$, a section of sheaf $\mathcal R$ over an element $U_{\alpha}$ of $\mathcal U$ admits an extension to $\tilde{U}_\alpha=\tilde{U}_{0,l} \times L_j$, where $\tilde{U}_{0,l}=\tilde{U}_0^1 \times \dots \times \tilde{U}_0^n \Subset U_0$ is a product domain such that each $\tilde{U}_0^i \subset \mathbb C$ ($1 \leq i \leq n$) is simply connected and has smooth boundary, and $U_{0,l}=\bar{V}_{0,k} \cap \tilde{U}_{0,l}$.
By part (1) of the proposition, over each $U_\alpha$ the sequence of sections $U_\alpha$ corresponding to (\ref{frseq}) is exact (there we can take product domain $\tilde{U}_{0,l}$ instead of polydisk $U_0$).
Hence, we have a sequence of cochain complexes
\begin{equation*}
\mathcal C^{\cdot}(\mathcal U,\mathcal F_{N}) \to \dots \to \mathcal C^{\cdot}(\mathcal U,\mathcal F_{1}) \to \mathcal C^{\cdot}(\mathcal U,\mathcal R) \to 0.
\end{equation*}
By Lemma \ref{reflem}(2) there exists a refinement $\mathcal U'$ of cover $\mathcal U$ of class $(P)$ of the same cardinality.
We have a commutative diagram with exact rows
\begin{equation*}
\bfig
\node a1(0,0)[\mathcal C^{\cdot}(\mathcal U',\mathcal F_{N})]
\node a2(500,0)[\dots]
\node a3(1000,0)[\mathcal C^{\cdot}(\mathcal U',\mathcal F_{1})]
\node a4(1600,0)[\mathcal C^{\cdot}(\mathcal U',\mathcal R) ]
\node a5(2050,0)[0]
\node b1(0,500)[\mathcal C^{\cdot}(\mathcal U,\mathcal F_{N})]
\node b2(500,500)[\dots]
\node b3(1000,500)[\mathcal C^{\cdot}(\mathcal U,\mathcal F_{1})]
\node b4(1600,500)[\mathcal C^{\cdot}(\mathcal U,\mathcal R) ]
\node b5(2050,500)[0]
\arrow[a1`a2;]
\arrow[a2`a3;]
\arrow[a3`a4;]
\arrow[a4`a5;]
\arrow[b1`b2;]
\arrow[b2`b3;]
\arrow[b3`b4;]
\arrow[b4`b5;]
\arrow[b1`a1;]
\arrow[b3`a3;]
\arrow[b4`a4;]
\efig
\end{equation*}
or, equivalently, the collection of commutative diagrams with exact rows
\begin{equation*}
\bfig
\node a1(0,0)[0]
\node a2(500,0)[\mathcal C^{\cdot}(\mathcal U',\mathcal R_i)]
\node a3(1200,0)[\mathcal C^{\cdot}(\mathcal U',\mathcal F_{i})]
\node a4(1950,0)[\mathcal C^{\cdot}(\mathcal U',\mathcal R_{i-1})]
\node a5(2500,0)[0,]
\node b1(0,500)[0]
\node b2(500,500)[\mathcal C^{\cdot}(\mathcal U,\mathcal R_i)]
\node b3(1200,500)[\mathcal C^{\cdot}(\mathcal U,\mathcal F_{i})]
\node b4(1950,500)[\mathcal C^{\cdot}(\mathcal U,\mathcal R_{i-1}) ]
\node b5(2500,500)[0]
\arrow[a1`a2;]
\arrow[a2`a3;]
\arrow[a3`a4;]
\arrow[a4`a5;]
\arrow[b1`b2;]
\arrow[b2`b3;]
\arrow[b3`b4;]
\arrow[b4`b5;]
\arrow[b2`a2;]
\arrow[b3`a3;]
\arrow[b4`a4;]
\efig
\end{equation*}
where $\mathcal R_i:=\Imag \varphi_{i}$ ($0 \leq i \leq N-1$), $\mathcal R_0:=\mathcal R$. 
In turn, each row yields the long exact sequence
\begin{multline*}
0 \to \Gamma(\bar{V}_k,\mathcal R_i) \to \Gamma(\bar{V}_k,\mathcal F_i) \to \Gamma(\bar{V}_k,\mathcal R_{i-1}) \to \\
H^1(\mathcal U,\mathcal R_i) \to H^1(\mathcal U,\mathcal F_i) \overset{\varphi^1_{i-1}}{\to} H^1(\mathcal U,\mathcal R_{i-1}) \overset{\psi_i^2}{\to} H^2(\mathcal U,\mathcal R_i) \to \dots, \quad 1 \leq i \leq N-1,
\end{multline*}
(and a similar one for $\mathcal U'$),
where $H^l(\mathcal U,\mathcal R_i):=\mathcal Z^l(\mathcal U,\mathcal R_i)/\mathcal B^l(\mathcal U,\mathcal R_i)$ are the \v{C}ech cohomology groups corresponding to cover $\mathcal U$. These sequences form the commutative diagram
\begin{equation*}
\bfig
\node a1(0,0)[\dots]
\node a2(500,0)[H^l(\mathcal U',\mathcal R_i)]
\node a3(1200,0)[H^l(\mathcal U',\mathcal F_i)]
\node a4(1950,0)[H^l(\mathcal U',\mathcal R_{i-1})]
\node a5(2700,0)[H^{l+1}(\mathcal U',\mathcal R_i)]
\node a6(3300,0)[\dots]
\node b1(0,500)[\dots]
\node b2(500,500)[H^l(\mathcal U,\mathcal R_i)]
\node b3(1200,500)[H^l(\mathcal U,\mathcal F_i)]
\node b4(1950,500)[H^l(\mathcal U,\mathcal R_{i-1})]
\node b5(2700,500)[H^{l+1}(\mathcal U,\mathcal R_i)]
\node b6(3300,500)[\dots]
\arrow[a1`a2;]
\arrow[a2`a3;]
\arrow[a3`a4;(\varphi^{l}_{i-1})' ]
\arrow[a4`a5;(\psi^{l+1}_{i})']
\arrow[a5`a6;]
\arrow[b1`b2;]
\arrow[b2`b3;]
\arrow[b3`b4;\varphi^{l}_{i-1} ]
\arrow[b4`b5;\psi^{ l+1}_{i}]
\arrow[b5`b6;]
\arrow[b2`a2;]
\arrow[b3`a3;\iota^l_i]
\arrow[b4`a4;\gamma^l_{i-1}]
\arrow[b5`a5;\gamma^{l+1}_{i}]
\efig
\end{equation*}
where $\iota_i^l$, $\gamma_{i-1}^l$, $\gamma_{i}^{l+1}$ are the corresponding refinement maps.

We have to show that given $\sigma \in H^l(\mathcal U,\mathcal R)$, $l \geq 1$, there exists a refinement $\mathcal W$ of cover $\mathcal U$ such that the image of $\sigma$ in $H^l(\mathcal W,\mathcal R)$ is zero. We construct this refinement using the following algorithm.

Suppose that there exists a non-zero $\sigma \in H^l(\mathcal U,\mathcal R_{i-1})$. 
Consider the following case:

(*) $\psi^{l+1} _{i}(\sigma)=0$. Then there exists $\eta \in H^l(\mathcal U,\mathcal F_i) $ such that $\sigma=\varphi^{l}_{i-1}(\eta)$. We have $\gamma^l_{i-1}(\sigma)=(\varphi^{l}_{i-1})'\bigl(\iota^l_i(\eta)\bigr)$. By Lemma \ref{cechlem0} $\iota^l_i\bigl(H^l(\mathcal U,\mathcal F_i)\bigr)=0$, hence the image of $\sigma$ by the refinement map $\gamma^l_{i-1}(\sigma)=0\in H^l(\mathcal U',\mathcal R_{i-1})$. 


We start with $\mathcal R_0=\mathcal R$ assuming that there exists a non-zero $\sigma \in H^{l}(\mathcal U,\mathcal R)$, $l \geq 1$. If case (*) occurs we set $\mathcal W:=\mathcal U'$. For otherwise,
there exists $2\le k\le s$ such that $(\psi_{k}^{l+k}\circ\dots\circ\psi_1^{l+1})(\sigma)=0\in H^{l+k}(\mathcal U,\mathcal R_k)$. (Indeed, assuming the opposite we obtain 
a non-zero element of $H^{l+s}(\mathcal U,\mathcal R_s)$; however, since the cardinality of $\mathcal U$ is $s$, we have $H^{l+s}(\mathcal U,\mathcal R_s)=0$, a contradiction.) Thus case (*) occurs for $\tilde \sigma:=(\psi_{k-1}^{l+k-1}\circ\dots\circ \psi_1^{l+1})(\sigma)$ instead of $\sigma$ which implies that the image of  $\tilde\sigma$ under the refinement map $H^{l+k-1}(\mathcal U,\mathcal R_{k-1})\rightarrow H^{l+k-1}(\mathcal U',\mathcal R_{k-1})$ is zero. Further, starting with cover $\mathcal U'$ (instead of $\mathcal U$)  and applying consequently case (*) to images of $(\psi_{p}^{l+p}\circ\dots\circ\psi_1^{l+1})(\sigma)$, $p=k-2,\dots, 1$, under the corresponding refinement maps we finally obtain the required refinement $\mathcal W$ of $\mathcal U$ such that the image of $\sigma$ under the refinement map $H^{l}(\mathcal U,\mathcal R_0)\rightarrow H^{l}(\mathcal W,\mathcal R_0)$ is zero.
\end{proof}

\subsection{Proof of Proposition \ref{cohcor}}
\label{cohproof}

The proof is based on the following lemma.

\begin{lemma}
\label{joinprop}
Let $U_0 \Subset \mathbb C^n$ be an open polydisk, and $K_1$, $K_2 \in \mathfrak Q$. Let $\mathcal R$ be an analytic sheaf over $U_0 \times (K_1 \cup K_2)$. Let $x_0 \in U_0$.

Suppose that for every $N \geq 1$ sheaf $\mathcal R$ admits free resolutions of length $N$ over $U_0 \times K_1$ and $U_0 \times K_2$. Then for any open subsets $L_1 \Subset K_1$, $L_2 \Subset K_2$ such that $L_i \in \mathfrak Q$ ($i=1,2$) there exists an open neighbourhood $V_0 \subset U_0$ of $x_0$ such that for every $N \geq 1$ sheaf $\mathcal R$ admits a free resolution of length $N$ over $V_0 \times (L_1 \cup L_2)$.
\end{lemma}

We prove Lemma \ref{joinprop} in the next subsection but now we use it in the proof of the proposition. 

\medskip

\begin{proof}[Proof of Proposition \ref{cohcor}] Let $U_0 \Subset \mathbb C^n$ be an open polydisk, $x_0 \in U_0$. Since sets $\bar{p}^{-1}(U_0)$ and $U_0 \times \hat{G}_{\mathfrak a}$ are biholomorphic (see~subsection \ref{charts}), it suffices to prove the proposition for a coherent sheaf $\mathcal A$ over $U_0 \times \hat{G}_{\mathfrak a}$.

By the definition of a coherent sheaf (see \eqref{coh0}), there exists a finite open cover of $\{x_0\} \times \hat{G}_{\mathfrak a}$ by sets $W_{0,i} \times L_i$, where $W_{0,i} \subset U_0$ is an open neighbourhood of $x_0$, $\cup_i L_i=\hat{G}_{\mathfrak a}$,
and for every $N \geq 1$ sheaf $\mathcal A$ admits free resolutions of  length $N$ over each $W_{0,i} \times L_i$.

By Lemma \ref{coverlem} there exists a collection of finite refinements $$\mathcal L^k(m)=\{L_j^k: L_j^k \in \mathfrak Q,~1 \leq j \leq m\}, \quad k \geq 1,$$ of open cover $\{L_i\}$ such that $L_j^{k+1} \Subset L_j^k$ for all $1 \leq j \leq m$, $k \geq 1$. 

Let $k=1$. We apply Lemma \ref{joinprop} to sheaf $\mathcal A$ with $K_1:=L_{m-1}^1$, $K_2:=L_{m}^1$, $L_1:=L_{m-1}^2$, $L_2:=L_{m-1}^2$ to obtain an open neighbourhood $V_0:=V_{0,m} \subset \cap_i W_{0,i}$ of $x_0$ such that for each $N \geq 1$ sheaf $\mathcal A$ has a free resolution of length $N$ over $V_{0,m} \times (L_{m-1}^2 \cup L_{m-1}^2)$. 

Next, we set $$\mathcal L^k(m-1):=\{L_1^k,\dots, L_{m-2}^k, \tilde L_{m-1}^k\},\quad \tilde L_{m-1}^k:=L_{m-1}^k \cup L_{m}^k,\quad k \geq 2.$$ 

Taking $k=2$ we  apply an argument similar to the above to the cover $\mathcal L^2(m-1)$ of $\hat{G}_{\mathfrak a}$ obtaining that for each $N \geq 1$ sheaf $\mathcal A$ has a free resolution of length $N$ over $V_{0,m-1} \times (L_{m-2}^3 \cup \tilde L_{m-1}^3)$ for some open neighbourhood $V_{0,m-1} \subset V_{0,m}$ of $x_0$. Then we define 
$$\mathcal L^k(m-2):=\{L_1^k,\dots,L_{m-3}^k, \tilde L_{m-2}^k\},\quad \tilde L_{m-2}^k:=L_{m-2}^k \cup \tilde L_{m-1}^k,\quad k \geq 3,\ \text{etc.}$$ 

After $m-1$ steps we obtain that there exists an open neighbourhood $V_{0,1} \subset \cap_i W_{0,i}$ of $x_0$ such that for each $N \geq 1$ sheaf $\mathcal A$ has a free resolution over $V_{0,1} \times \hat{G}_{\mathfrak a}$, as required. 
\end{proof}

\subsection{Proof of Lemma \ref{joinprop}}

We will use the following notation.

Let $M_{l \times k}(\mathbb C)$ be the space of $l \times k$ matrices $C=(c_{ij})$ with entries $c_{ij} \in \mathbb C$ endowed with norm $|C|:=\max\{|c_{ij}|\}_{i,j=1}^{l,k}$. We set $M_{k}(\mathbb C):=M_{k \times k}(\mathbb C)$.

 Let $GL_k(\mathbb C) \subset M_{k}(\mathbb C)$ be the group of invertible matrices.
We denote by $I=I_k \in GL_k(\mathbb C)$ the identity matrix.

Let $U_0 \subset \mathbb C^n$ be an open polydisk, $K \in \mathfrak Q$; set $U:=U_0 \times K$.
The space $\mathcal O(U,M_k(\mathbb C))$ of holomorphic $M_k(\mathbb C)$-valued functions is endowed with norm
\begin{equation*}
\|F\|_U:=\sup_{x \in U} |F(x)|, \quad F \in \mathcal O(U,M_k(\mathbb C)).
\end{equation*}
The subset $\mathcal O(U,GL_k(\mathbb C)) \subset \mathcal O(U,M_k(\mathbb C))$ of holomorphic $GL_k(\mathbb C)$-valued maps on $U$ has the induced topology of uniform convergence on compact subsets of $U$ (see~Lemma \ref{exhlem}(2)).

The identity map $(z,\omega) \rightarrow I$, $(z,\omega) \in U$, will be denoted also by $I$.

\begin{lemma}
\label{cartanlem0}
Let
 $U':=U_0 \times K'$, $U'':=U_0 \times K''$, where $K'$, $K'' \in \mathfrak Q$. 
 Suppose that $H \in \mathcal O\bigl(U' \cap U'',GL_k(\mathbb C) \bigr)$ belongs to the connected component of the identity map $I$ in $\mathcal O(U' \cap U'',GL_k(\mathbb C))$. 
 
 Then for any open polydisk $V_0 \Subset U_0$ and open subsets $L' \Subset K'$, $L'' \Subset K''$ there exists a  function $H' \in \mathcal O\bigl(V' \cup V'',GL_k(\mathbb C)\bigr)$, where $V':=V_0 \times L'$, $V'':=V_0 \times L''$,
such that  
$H'|_{V' \cap V''}=H|_{V' \cap V''}$.
\end{lemma}

\begin{proof} 
We may assume without loss of generality that polydisks $V_0$, $U_0$ are centered at the origin $0 \in \mathbb C^n$.

First, suppose that $\|I-H\|_{V' \cap V''}<\frac{1}{2}$. Then we can define
$F:=\log H=-\sum_{j=1}^\infty\frac{(I-H)^j}{j} \in \mathcal O\bigl(V' \cap V'',M_k(\mathbb C) \bigr) \cap C\bigl(\bar{V}' \cap \bar{V}'',M_k(\mathbb C) \bigr)$.
Let us show that there exists a function $F' \in \mathcal O\bigl(V' \cup V'',M_k(\mathbb C)$ such that
$F'|_{V' \cap V''}=F|_{V' \cap V''}$. Indeed, we can expand the $C(\bar{L}' \cap \bar{L}'',M_k(\mathbb C))$-valued holomorphic function $F(z,\cdot)$ in the Taylor series about $0$,
\begin{equation*}
F(z,\eta)=\sum_{m=0}^\infty b_m(\eta)z^m, \quad z \in \tilde V_0, \quad \eta \in \bar{L}' \cap \bar{L}'',
\end{equation*}
where $b_m \in C(\bar{L}' \cap \bar{L}'',M_k(\mathbb C) )$ and $\tilde V_0$ is an open neighbourhood of $\bar V_0$.
Note that $\bar{L}' \cup \bar{L}''$ is compact (as a closed subspace of compact space $\hat{G}_{\mathfrak a}$), and hence is normal. 
Therefore, using the Tietze--Urysohn extension theorem, we can extend each $b_m$ to a function $\tilde{b}_m \in C(\bar{L}' \cup \bar{L}'',M_k(\mathbb C))$ such that $\sup_{\omega \in \bar{L}' \cap \bar{L}''}|b_m(\omega)|=\sup_{\omega \in \bar{L}' \cup \bar{L}''}|\tilde{b}_m(\omega)|$. Then we
define 
$$F'(z,\omega):=\sum_{m=0}^\infty \tilde{b}_m(\omega)z^m, \quad z \in V_0, \quad \omega \in L' \cup L''.$$
(Since the above series converges uniformly on relatively compact subsets of $\tilde V_0$, $F' \in \mathcal O\bigl(V' \cup V'',M_k(\mathbb C)$ and satisfies the required condition.) 

Now, we set
$H':=\exp(F') \in \mathcal O\bigl(V' \cup V'',GL_k(\mathbb C)\bigr)$ completing the proof of the lemma in this case.

Further, let $H \in \mathcal O\bigl(U' \cap U'',GL_k(\mathbb C) \bigr)$ be an arbitrary $GL_k(\mathbb C)$-valued bounded holomorphic map belonging to the connected component of the identity map $I$ of $\mathcal O\bigl(U' \cap U'',GL_k(\mathbb C) \bigr)$. 


Let us show that $H|_{V' \cap V''}$ can be presented in the form
\begin{equation}
\label{pr1}
H|_{V' \cap V''}=H^1 \cdots H^l,
\end{equation}
where each $H^i \in \mathcal O(V' \cap V'',GL_k(\mathbb C))$, $1 \leq i \leq l$, satisfies
\begin{equation}
\label{pr2}
\|I-H^i\|_{V' \cap V''}<\frac{1}{2}.
\end{equation}
In fact, since $H$ belongs to the connected component of the identity map $I$, there exists a continuous path $H_t \in \mathcal O\bigl(U' \cap U'',GL_k(\mathbb C) \bigr)$ ($t \in [0,1]$) such that $H_0=I$, $H_1=H$.
Consider a partition $0=t_0<t_1<\dots<t_l=1$ of the unit interval $[0,1]$, and define
$$H^i(z,\omega)=H^{-1}_{t_{i-1}}(z,\omega)H_{t_i}(z,\omega), \quad (z,\omega) \in V' \cap V'', \quad 1 \leq i \leq l,$$
which gives us identity (\ref{pr1}).
Provided that $\max_{1 \leq i \leq l-1}|t_{i+1}-t_i|$ is sufficiently small, inequality (\ref{pr2}) holds for all $1 \leq i \leq l$.

Now, according to the first case there exist $(H^{i})' \in \mathcal O(V' \cup V'',GL_k(\mathbb C))$ such that $(H^{i})'|_{V' \cap V''}=H^i|_{V' \cap V''}$. We define $H':=(H^{1})'\cdots (H^l)'$.
\end{proof}

\begin{corollary}
\label{cartanlem}
 
In the notation of Lemma \ref{cartanlem0}, for any open polydisk $V_0 \Subset U_0$ and open subsets $L' \Subset K'$, $L'' \Subset K''$ there exist functions $h' \in \mathcal O\bigl(V',GL_k(\mathbb C)\bigr)$, $h'' \in \mathcal O\bigl(V'',GL_k(\mathbb C)\bigr)$ such that
$$H=h'h'', \quad \text{ on }V' \cap V''.$$
\end{corollary}

\begin{proof}
Let $H' \in \mathcal O(V' \cup V'',GL_k(\mathbb C))$ be as in Lemma \ref{cartanlem0}. Since $H'|_{V' \cap V''}=H|_{V' \cap V''}$, we can choose $h':=H'|_{V'}$, $h'':=I$.
\end{proof}

\begin{lemma}
\label{matrixlem}
Any analytic homomorphism $\varphi:\mathcal O|_{U}^{k} \rightarrow \mathcal O|_{U}^{l}$ is determined by a holomorphic function $\Phi \in \mathcal O\bigl(U,M_{l \times k}(\mathbb C)\bigr).$ 
\end{lemma}

The proof follows directly from the definitions.

\begin{defin}[see~\cite{Lemp}]
\label{adef2}
Let $\mathcal R$, $\mathcal B_i$, $1 \leq i \leq N$, be analytic sheaves over $U$.
We say that a sequence 
\begin{equation}
\label{aseq1}
\mathcal B_N \to \dots \to \mathcal B_2\to \mathcal B_1 \to \mathcal R \to 0
\end{equation}
is \textit{completely exact} if for any $m \geq 1$ 
the sequence of sections
\begin{equation*}
\Gamma(U,\Hom(\mathcal O^{m},\mathcal B_N)) \to \dots \\ \to \Gamma(U,\Hom(\mathcal O^{m},\mathcal B_1))\to \Gamma(U,\Hom(\mathcal O^{m},\mathcal R)) \to 0
\end{equation*}
or, equivalently,
\begin{equation}
\label{equivcond}
\Gamma(U,\mathcal B^m_N) \to \dots \to \Gamma(U,\mathcal B^m_1)\to \Gamma(U,\mathcal R^m) \to 0,
\end{equation}
is exact.

Here $\mathcal B_i^m$ and $\mathcal R^m$ stand for the direct product of $m$ copies of $\mathcal B_i$ and $\mathcal R$, respectively, and $\Hom(\mathcal O^{m},\mathcal B_i)$, $\Hom(\mathcal O^{m},\mathcal R)$ are the sheaves of germs of analytic homomorphisms $\mathcal O^{m} \rightarrow \mathcal B_i$, $\mathcal O^{m} \rightarrow \mathcal R$, respectively.
\end{defin}

Note that if sequence (\ref{equivcond}) is exact for $m=1$, then it is exact for all $m \geq 1$. 

The next two lemmas are due to Lempert \cite{Lemp}.

\begin{lemma}
\label{lem65}
Let $\mathcal B$, $\mathcal C$ be analytic sheaves on $U$.
If sequence $\mathcal B \overset{\gamma}{\rightarrow} \mathcal C \rightarrow 0$ is completely exact, and $\varphi:\mathcal O^k|_U \rightarrow \mathcal C$ is an analytic homomorphism, then there is an analytic homomorphism $\psi:\mathcal O^k|_U \rightarrow \mathcal B$ such that $\varphi=\gamma \psi$.
\end{lemma}
\begin{proof}
We can take $\psi$ in the preimage of $\varphi$ under the surjective homomorphism
$$\bar{\gamma}:\Gamma(U,\Hom(\mathcal O^k,\mathcal B)) \rightarrow \Gamma(U,\Hom(\mathcal O^k,\mathcal C))$$ induced by $\gamma$ (see~Definition \ref{adef2}).
\end{proof}

\begin{lemma}[Three lemma]
\label{lem7}
Let $\mathcal A$, $\mathcal B$ and $\mathcal C$ be analytic sheaves on $U$.
Suppose that sequence $$0 \to \mathcal A \overset{\beta}{\to} \mathcal B \overset{\gamma}{\to} \mathcal C \to 0$$ is completely exact. 
If two among $\mathcal A$, $\mathcal B$ and $\mathcal C$ have free resolutions of length $N+n$, where $n:=\dim_{\mathbb C} U_0$, $N \geq n+2$, then the third has a free resolution of length $N-n-1$.
\end{lemma}

For the sake of completeness, we provide the proof of the lemma in the Appendix; but now we will use it in the proof of Lemma \ref{joinprop}.

\begin{proof}[Proof of Lemma \ref{joinprop}] We denote $U_1:=U_0 \times K_1$, $U_2:=U_0 \times K_2$. Let $N \geq n+1$. Consider free resolutions of $\mathcal R$ of length $M \geq 4N $,
\begin{equation}
\label{frres1}
\mathcal O^{k_{M,i}}|_{U_i} \to \dots \to
\mathcal O^{k_{1,i}}|_{U_i} \overset{\alpha_{i}}{\to} \mathcal R|_{U_i} \to 0, \quad i=1,2.
\end{equation}
Consider the end portions of (\ref{frres1}):
\begin{equation}
\label{s8}
\mathcal O^{k_i}|_{U_i} \overset{\alpha_{i}}{\to} \mathcal R|_{U_i} \to 0, \quad i=1,2.
\end{equation}
Let $U:=U_0 \times (K_1 \cup K_2)$.
We denote by $\pi_i:\mathcal O^{k_1}|_U \oplus \mathcal O^{k_2}|_U \rightarrow \mathcal O^{k_i}|_U$, $i=1,2$, the natural projection homomorphisms. 

First, let us show that there exists an injective analytic homomorphism
$H:\mathcal O^{k_1}|_U \oplus \mathcal O^{k_2}|_U \rightarrow \mathcal O^{k_1}|_U \oplus \mathcal O^{k_2}|_U$
such that $\alpha_1\pi_1 H=\alpha_2\pi_2$.
By Proposition \ref{vanprop}(1) sequence (\ref{frres1}) truncated to the $N$-th term  (and, hence, sequence (\ref{s8})) is completely exact. By Lemma \ref{lem65} we can factor
$\alpha_1=\alpha_2\psi$, $\alpha_2=\alpha_1\varphi$ on $U_1 \cap U_2$
for some analytic homomorphisms
$\psi:\mathcal O^{k_1}|_{U_1 \cap U_2} \rightarrow \mathcal O^{k_2}|_{U_1 \cap U_2} $,
$\varphi:\mathcal O^{k_2}|_{U_1 \cap U_2}  \rightarrow \mathcal O^{k_1}|_{U_1 \cap U_2} $.
Now, identifying sheaf homomorphisms $\psi$, $\varphi$ with the holomorphic matrix functions that determine them (see~Lemma \ref{matrixlem}), we define
\begin{equation*}
H=\left(
\begin{array}{cc}
I_{k_1}&\varphi \\
0& I_{k_2}
\end{array}
\right)
\left(
\begin{array}{cc}
I_{k_1}&0 \\
\psi& I_{k_2}
\end{array}
\right)^{-1} \in \mathcal O(U_1 \cap U_2,GL_k(\mathbb C)),
\end{equation*}
where $k:=k_1+k_2$.
It is immediate that $\alpha_1\pi_1H=\alpha_2\pi_2$. 
The map $H$ belongs to the connected component of the identity map in $\mathcal O(U_1 \cap U_2,GL_k(\mathbb C))$. Indeed, consider a path $H_{t} \in \mathcal O(U_1 \cap U_2,GL_k(\mathbb C))$ ($t \in [0,1]$),
\begin{equation*}
H_{t}:=\left(
\begin{array}{cc}
I_{k_1}&t\varphi \\
0& I_{k_2}
\end{array}
\right)
\left(
\begin{array}{cc}
I_{k_1}&0 \\
t\psi& I_{k_2}
\end{array}
\right)^{-1} ,
\end{equation*}
so that $H_{0}=I_k$, $H_{1}=H$.

Next, let $L_i \Subset K_i$, $L_i \in \mathfrak Q$ ($i=1,2$) and $V_0 \Subset U_0$ be an open polydisk, $x_0 \in V_0$. Let $L_i^m \in \mathfrak Q$ ($i=1,2$), $m \geq 1$, be the collection of open subsets of $K_i$ such that $L_i \Subset L_i^{m+1} \Subset L_i^m \Subset K_i$ for all $m \geq 1$ ($i=1,2$) obtained in Lemma \ref{exhlem}(3). 

Let $\{V_0^m\}$ be a collection of open polydisks such that $V_0 \Subset V_0^{m+1} \Subset V_0^m \Subset U_0$ for all $m \geq 1$.

We set $V_i^m:=V^m_0 \times L_i^m$,  $V_i:=V_0 \times L_i$ ($i=1,2$), $m \geq 1$.

We now amalgamate the free resolutions of $\mathcal R$ over $V_1\cup V_2$.

Let $m=1$.
By Corollary \ref{cartanlem} there exist functions $h_i \in \mathcal O(V^1_i,GL_k(\mathbb C))$ ($i=1,2$) such that $H=h_1h_2$ on $V^1_1 \cap V^1_2$.
Since $\alpha_1\pi_1 H=\alpha_2\pi_2$, 
the sheaf homomorphisms
\begin{equation*}
\alpha_1\pi_1h_1:\mathcal O^{k_1}|_{V^1_1} \oplus \mathcal O^{k_2}|_{V^1_1} \to \mathcal R|_{V^1_1} \to 0,
\end{equation*}
\begin{equation*}
\alpha_2\pi_2h_2^{-1}:\mathcal O^{k_1}|_{V^1_2} \oplus \mathcal O^{k_2}|_{V^1_2} \to \mathcal R|_{V^1_2} \to 0
\end{equation*}
coincide over $V^1_1 \cap V^1_2$; they induce an analytic homomorphism
\begin{equation*}
\alpha:\mathcal O^{k_1}|_{V^1_1 \cup V^1_2} \oplus \mathcal O^{k_2}|_{V^1_1 \cup V^1_2} \to \mathcal R|_{V^1_1 \cup V^1_2}.
\end{equation*}
Let $\mathcal R_1:=\Ker \alpha$. The sequence
\begin{equation*}
0 \to \mathcal R_1|_{V^1_1 \cup V^1_2} \to \mathcal O^{k_1}|_{V^1_1 \cup V^1_2} \oplus \mathcal O^{k_2}|_{V^1_1 \cup V^1_2} \overset{\alpha}{\to} \mathcal R|_{V^1_1 \cup V^1_2} \to 0
\end{equation*}
is completely exact over sets $V_1^1$ and $V_2^1$ since sequences (\ref{s8}) are. By Lemma \ref{lem7} the analytic sheaf $\mathcal R_1$ has free resolutions over $V_1^1$ and $V_2^1$ (of length $4N-2n-1$) because two other sheaves have. 

Provided that $M$ was chosen sufficiently large, we can repeat this construction $N-1$ times over subsets $V_1^m$, $V_2^m$, $1 \leq m \leq N-1$, obtaining in the end a free resolution of $\mathcal R$  over $V_1 \cup V_2$ having length $N$. Since $V_0$, $L_1$, $L_2$ and $N$ were arbitrary, the required result follows.
\end{proof}

\section{Proof of Theorem \ref{holmapthm}}

\begin{proof}
Let $y_0:=F(z_0)$ for some $z_0 \in M$. Then $y_0 \in \hat{X}_{H_0}$ for some $H_0 \in \Upsilon$ (see~subsection \ref{structsection}). Also, $y_0$ is contained in a coordinate chart $\hat{\Pi}(U_0,K) \subset c_{\mathfrak a}X$. In what follows, we identify $\hat{\Pi}(U_0,K)$ with $U_0 \times K$, see~subsection \ref{charts}, so that $y_0=(x_0,\eta_0)$ for some $x_0 \in U_0$, $\eta_0 \in K$.  Let $\pi_K: U_0\times K\to K$ be the natural projection. Then for each $h \in C(K)$ the pullback $h_K:=(\pi_K)^*h\in\mathcal O(U_0,K)$ and is constant on subsets $U_0\times\{\eta\}$ for all $\eta\in K$. Since $F$ is a holomorphic map, $F^*h_K$ is holomorphic on open subset $F^{-1}(U_0 \times K) \subset M$. Since the complex conjugate $\bar{h}_K$ of $h_K$ also belongs to $\mathcal O(U_0,K)$, the function $F^*\bar{h}_K=\overline{F^*h_K}$ is holomorphic on $F^{-1}(U_0 \times K)$ as well. Therefore, $F^*h_K$ must be locally constant. 
Let $W \subset M$ be the connected component of $F^{-1}(U_0 \times K)$ containing $z_0$; then $F^*h_K \equiv h(\eta_0)$ on $W$. 

Now, let us show that $F(W) \subset U_0 \times \{\eta_0\} \subset \hat{X}_{H_0}$. Indeed, there exist open subsets $L_\lambda \subset K$ ($\lambda \in \Lambda$) such that $\eta_0 \in L_\lambda$ for all $\lambda$ and $\cap_{\lambda \in \Lambda} L_\lambda=\{\eta_0\}$. Since $\hat{G}_{\mathfrak a}$ is a compact space and each subset $L_\lambda$ is open in $\hat{G}_{\mathfrak a}$, for every $\lambda$ there exists a continuous partition of unity subordinate to the open cover $\{L_\lambda,\hat{G}_{\mathfrak a} \setminus \{\eta_0\}\}$ of $\hat{G}_{\mathfrak a}$. We denote by $h_\lambda \in C(K)$ the restriction to $K$ of the element of the partition of unity with support in $L_\lambda$. Then
$0 \leq h_\lambda \leq 1$, $h_\lambda(\eta_0)=1$, $h_\lambda(\eta)=0$ on $K \setminus L_\lambda$. Since $F^*(h_\lambda)_K \equiv h_\lambda(\eta_0)=1$ for all $\lambda$, we obtain that $F(W) \subset U_0 \times L_\lambda$ for all $\lambda$; hence, $F(W) \subset U_0 \times \cap_{\lambda \in \Lambda} L_\lambda=U_0 \times \{\eta_0\} \subset \hat{X}_{H_0}$.

We have established that every point in $M$ has a neighbourhood $W$ such that $F(W) \subset \hat{X}_H$ for some $H \in \Upsilon$. Since $\hat{X}_{H_1} \cap \hat{X}_{H_2}=\varnothing$ if $H_1 \neq H_2$ and $M$ is connected, the latter implies that $F(M) \subset \hat{X}_{H_0}$ for a certain $H_0$ and completes the proof of the theorem.
\end{proof}

\section{Appendix}

The proof of Proposition \ref{frechetprop} essentially repeats the proof of an analogous result for coherent analytic sheaves on complex manifolds, see, e.g., \cite{GR}.

The proofs of the other results of this section follow closely the arguments in \cite{Lemp}.

\subsection{Proof of Proposition \ref{frechetprop}}


First, let $\mathcal A$ be a coherent subsheaf of $\mathcal O^k$ and let $U \in \mathfrak B$ (see~(\ref{base2})). 
By Lemma \ref{exhlem}(2) there exist open sets $V_k \in \mathfrak B$ such that $V_{k} \Subset V_{k+1} \Subset U$ for all $k$, and $U=\cup_k V_k$.
We endow space $\Gamma(U,\mathcal A)$ of sections of sheaf $\mathcal A$ over $U$ with the topology of uniform convergence on $\bar{V}_k$ for all $k$. Then $\Gamma(U,\mathcal A)$ becomes a metrizable vector space. We have to show that space $\Gamma(U,\mathcal A)$ is complete, i.e., it is a Fr\'{e}chet space. 

It is easy to see that space $\Gamma(U,\mathcal O^k)$ endowed with such topology is complete.
Since $\mathcal A$ is coherent, we may assume that there exists a free resolution (\ref{coh0}) of $\mathcal A$ over $U$ of length $4N$, $N>n:=\dim_{\mathbb C}X_0$. Therefore, we have a short exact sequence
\begin{equation*}
0 \to \Ker \varphi \overset{\iota}{\to}  \mathcal O^{m}|_U \overset{\varphi}{\to} \mathcal A|_U \to 0,
\end{equation*}
where $\iota$ denotes the inclusion. In the proof of Proposition \ref{vanprop}(1) we have shown that the sequence of sections
\begin{equation}
\label{exseq100}
0 \to \Gamma(U,\Ker \varphi) \overset{\bar{\iota}}{\to} \Gamma(U,\mathcal O^{m}) \overset{\bar{\varphi}}{\to} \Gamma(U,\mathcal A) \to 0
\end{equation}
is exact (see~Lemmas \ref{lem3comp} and \ref{sectionslem}). By our assumption $\Gamma(U,\mathcal A) \subset \Gamma(U,\mathcal O^k)$. By Lemma \ref{matrixlem} the $\Gamma(U,\mathcal O)$-module homomorphism $\bar{\varphi}:\Gamma(U,\mathcal O^m) \rightarrow \Gamma(U,\mathcal O^k)$ is determined by a $k \times m$ matrix with entries in $\mathcal O(U)$, hence it is continuous; further, $\bar{\iota}$ is continuous. Since sequence (\ref{exseq100}) is exact, $\Gamma(U,\Ker \varphi) \cong \Ker \bar{\varphi}$, hence $\Gamma(U,\Ker \varphi)$ is closed. Therefore, $\Gamma(U,\mathcal A)$, being a quotient of a complete space by its closed subspace, is a complete space.

We note that by the open mapping theorem the topology on $\Gamma(U,\mathcal A)$  coincides with the quotient topology determined by (\ref{exseq100}).

\medskip

Now, let $\mathcal A$ be an arbitrary coherent sheaf on $c_{\mathfrak a}X$. Similarly, we have a free resolution (\ref{coh0}) of $\mathcal A$ over a neighbourhood $U$ of length $4N$, $N>n$, which yields a short exact sequence of sheaves
\begin{equation}
\label{sectexseq2}
0 \to \Ker \varphi \overset{\iota}{\to} \mathcal O^{m}|_U \overset{\varphi}{\to} \mathcal A|_U \to 0
\end{equation}
and an exact sequence of sections
\begin{equation}
\label{sectexseq}
0 \to \Gamma(U,\Ker \varphi) \overset{\bar{\iota}}{\to} \Gamma(U,\mathcal O^{m}) \overset{\bar{\varphi}}{\to} \Gamma(U,\mathcal A) \to 0.
\end{equation}
Using Lemma \ref{lem7} (Three lemma), we obtain that $\Ker \varphi$ is a coherent subsheaf of $\mathcal O^{m}|_U$,
so by the previous part the subspace $\Gamma(U,\Ker \varphi) \subset \Gamma(U,\mathcal O^{m})$ is closed. We introduce in $\Gamma(U,\mathcal A)$ the quotient topology defined by (\ref{sectexseq}) which makes it a complete (i.e., Fr\'{e}chet) space
and also implies the last assertion of the proposition concerning the family of semi-norms determining the topology in $\Gamma(U,\mathcal A)$.

Let us show that thus defined topology on $\Gamma(U,\mathcal A)$ does not depend on the choice of resolution (\ref{sectexseq2}). Suppose that there is another resolution
\begin{equation*}
0 \to \Ker \varphi' \overset{\iota}{\to} \mathcal O^{m'}|_U \overset{\varphi'}{\to} \mathcal A|_U \to 0.
\end{equation*}
By Lemma \ref{lem5} there is a homomorphism $\psi:\mathcal O^{m}|_U \rightarrow \mathcal O^{m'}|_U$ such that the diagram of exact sequences of sheaves
\begin{equation*}
\bfig
\node n1(0,0)[\mathcal O^{m'}|_U]
\node n2(800,0)[\mathcal A|_U]
\node n3(1600,0)[0]
\node m1(0,500)[\mathcal O^{m}|_U]
\node m2(800,500)[\mathcal A|_U]
\node m3(1600,500)[0]
\arrow[n1`n2;\varphi']
\arrow[n2`n3;]
\arrow[m1`m2;\varphi]
\arrow[m2`m3;]
\arrow[m1`n1;\psi]
\arrow[m2`n2;\lambda]
\efig
\end{equation*}
is commutative. Therefore, we have a commutative diagram
\begin{equation*}
\bfig
\node n1(0,0)[\Gamma(U,\mathcal O^{m'})]
\node n2(800,0)[\Gamma(U,\mathcal A)]
\node n3(1600,0)[0]
\node m1(0,500)[\Gamma(U,\mathcal O^{m})]
\node m2(800,500)[\Gamma(U,\mathcal A)]
\node m3(1600,500)[0]
\arrow[n1`n2;\bar{\varphi}']
\arrow[n2`n3;]
\arrow[m1`m2;\bar{\varphi}]
\arrow[m2`m3;]
\arrow[m1`n1;\bar{\psi}]
\arrow[m2`n2;\bar{\lambda}]
\efig
\end{equation*}
of exact sequences of sections.
By our construction $\bar{\varphi}$, $\bar{\varphi}'$ are continuous and surjective, $\bar{\psi}$ is continuous as a homomorphism of sections of free sheaves. By the open mapping theorem the preimage of an open set by $\bar{\lambda}^{-1}=\bar{\varphi}\circ(\bar{\psi})^{-1}\circ(\bar{\varphi}')^{-1}$ is open, so $\bar{\lambda}$ is continuous and, hence, it is a homeomorphism.

\medskip

Finally, let $\gamma:\mathcal A \rightarrow \mathcal B$ be an analytic homomorphism. Let us show that $\gamma$ is continuous. Analogously to the previous part applying Lemma \ref{lem5} we obtain a commutative diagram of exact sequences of sheaves which yields a commutative diagram of exact sequences
\begin{equation*}
\bfig
\node n1(0,0)[\Gamma(U,\mathcal O^{m'})|_U]
\node n2(800,0)[\Gamma(U,\mathcal B)]
\node n3(1600,0)[0.]
\node m1(0,500)[\Gamma(U,\mathcal O^{m})]
\node m2(800,500)[\Gamma(U,\mathcal A)]
\node m3(1600,500)[0]
\arrow[n1`n2;\bar{\varphi}']
\arrow[n2`n3;]
\arrow[m1`m2;\bar{\varphi}]
\arrow[m2`m3;]
\arrow[m1`n1;\bar{\psi}]
\arrow[m2`n2;\bar{\gamma}]
\efig
\end{equation*}
As before, the continuity of $\bar{\gamma}$ can be deduced from the continuity of the other homomorphisms in the diagram. This completes the proof of the proposition.

\subsection{Proof of Lemma \ref{lem7}}


We will need the following lemmas.

\begin{lemma}
\label{lem5}
Let $\mathcal A$ be an analytic sheaf on $U$ that admits a free resolution of length $N$
\begin{equation}
\label{aseq4}
\mathcal F_N|_{U} \overset{\varphi_{N-1}}{\to} \dots \overset{\varphi_{2}}{\to} \mathcal F_2|_{U} \overset{\varphi_1}{\to} \mathcal F_1|_{U} \overset{\varphi_0}{\to} \mathcal A \to 0.
\end{equation}
Given a completely exact sequence of analytic sheaves $\mathcal B_i$ on $U$, $0 \leq i \leq N$,
\begin{equation}
\label{seq4}
\mathcal B_N \overset{\beta_{N-1}}{\to} \dots \to \mathcal B_2 \overset{\beta_1}{\to} \mathcal B_1 \overset{\beta_0}{\to} \mathcal B_0 \to 0
\end{equation}
a sheaf homomorphism $\Phi_0:\mathcal A \rightarrow \mathcal B_0$ can be extended to a homomorphism $\Phi_j: \mathcal F_j|_{U} \rightarrow \mathcal B_j$ ($0 \leq j \leq N$) of sequences (\ref{aseq4}), (\ref{seq4}).
\end{lemma}


\begin{proof}
The proof is by induction. We put $\varphi_{-1}:=0$, $\beta_{-1}:=0$. Suppose that for $0 \leq j \leq r$, $r \leq N-1$ the homomorphisms $\Phi_j: \mathcal F_j|_{U} \rightarrow \mathcal B_j$ have been constructed, so that $\Phi_{j-1}\varphi_{j-1}=\beta_{j-1}\Phi_j$. If $r=N-1$, then we are done. For $r<N-1$
we have $\beta_{r-1}(\Phi_r \varphi_r)=\Phi_{r-1}\varphi_{r-1}\varphi_r=0$. The sequence 
\begin{equation*}
\Gamma(U, \Hom(\mathcal F_{r+1},\mathcal B_{r+1})) \to \dots \to \Gamma(U, \Hom(\mathcal F_{r+1},\mathcal B_{0})) \to 0
\end{equation*}
is exact since (\ref{seq4}) is completely exact (see~Definition \ref{adef2}), hence there is a homomorphism
$\Phi_{r+1} \in \Gamma(U, \Hom(\mathcal F_{r+1},\mathcal B_{r+1}))$ such that $\Phi_r \varphi_r=\beta_r \Phi_{r+1}$ over $U$, as required.
\end{proof}

\begin{lemma}
\label{lem55}
Given a free resolution (\ref{aseq4}) of an analytic sheaf $\mathcal A$ on $U$ of length $N$ the sheaf $\Ker \varphi_{n-1}=\Imag \varphi_n$ on $U$, $1 \leq n \leq N-1$, has a free resolution of length $N-n$.
\end{lemma}
\begin{proof}
Follows immediately from Definition \ref{adef23} of a free resolution of an analytic sheaf.
\end{proof}

\begin{lemma}
\label{lem6}
Let $\mathcal A_0$ be an analytic sheaf over $U$. Suppose that for a given $N \geq 1$ there exists a completely exact sequence 
of analytic sheaves $\mathcal A_i$ on $U$, $1 \leq i \leq 2N+2$,
\begin{equation}
\label{dseq1}
\mathcal A_{2N+2} \overset{\alpha_{M-1}}{\to} \dots \overset{\alpha_1}{\to} \mathcal A_1 \overset{\alpha_0}{\to} \mathcal A_0 \to 0
\end{equation}
such that sheaves $\mathcal A_i$, $1 \leq i \leq 2N+2$, have free resolutions of length $n+N$, where $n:=\dim_{\mathbb C}U_0$. 
 Then $\mathcal A_0$ has a free resolution of length $N$.
\end{lemma}
\begin{proof}
Let $M:=2N+2$.

(1) First, we construct a completely exact sequence of length $M-2$ of the form
\begin{equation}
\label{seq2}
 \mathcal B_{M-2} \overset{\beta_{M-3}}{\to} \dots \overset{\beta_2}{\to} \mathcal B_2 \overset{\beta_1}{\to} \mathcal B_1 \overset{\varepsilon_0}{\to} \mathcal A_0 \to 0,
\end{equation}
where $\mathcal B_1=\mathcal O^k|_U$ for some $k \geq 0$ is a free sheaf and $\mathcal B_k$, $2 \leq k \leq M-2$, are analytic sheaves on $U$ having free resolutions of length $N+n-1$.
Let 
\begin{equation*}
 \mathcal F_{n+N,k} \to \dots \to \mathcal F_{1,k} \overset{\omega_k}{\to} \mathcal A_k \to 0
\end{equation*}
be a free resolution of $\mathcal A_k$, $1 \leq k \leq M$. By Lemma \ref{lem5} there exist analytic homomorphisms $\psi_k$ such that the diagram
\begin{equation}
\label{s3}
\bfig
\node e1(-500,0)[\mathcal A_M]
\node e2(-500,500)[0]
\node e3(-500,-500)[\mathcal F_{1,M}]
\node a31(700,500)[0]
\node a41(1400,500)[0]
\node a12(100,0)[\dots]
\node a32(700,0)[\mathcal A_2]
\node a42(1400,0)[\mathcal A_1]
\node a52(2100,0)[\mathcal A_0]
\node a62(2700,0)[0]
\node a13(100,-500)[\dots]
\node a33(700,-500)[\mathcal F_{1,2}]
\node a43(1400,-500)[\mathcal F_{1,1}]
\arrow[e1`e2;]
\arrow[e3`e1;\omega_M]
\arrow[e1`a12;\alpha_{M-1}]
\arrow[e3`a13;\psi_{M-1}]
\arrow[a12`a32;\alpha_2]
\arrow[a32`a31;]
\arrow[a42`a41;]
\arrow[a32`a42;\alpha_1]
\arrow[a42`a52;\alpha_0]
\arrow[a52`a62;]
\arrow[a13`a33;\psi_2]
\arrow[a33`a43;\psi_1]
\arrow[a33`a32;\omega_2]
\arrow[a43`a42;\omega_1]
\efig
\end{equation}
is commutative. Let us show that the sequence
\begin{equation}
\label{s4}
\mathcal F_{1,M} \oplus \Ker \omega_{M-1} \overset{\beta_{M-1}}{\to} 
\dots \overset{\beta_2}{\to} \mathcal F_{1,2} \oplus \Ker \omega_1 \overset{\beta_1}{\to} \mathcal F_{1,1} \overset{\beta_0}{\to} \mathcal A_0 \to 0
\end{equation}
truncated to term $\mathcal F_{1,M-2} \oplus \Ker \omega_{M-3}$
is completely exact. Here $\beta_0:=\alpha_0\omega_1$, $\beta_1:=\psi_1 - \iota_1$, where $\iota_k:\Ker \omega_k \hookrightarrow \mathcal F_{1,k}$ is an inclusion, and $\beta_k=(\iota_k \oplus \psi_{k-1})(\psi_k-\iota_k)$, $k \geq 2$.

Indeed, we apply to (\ref{s3}) and (\ref{s4}) left exact functor $\Gamma(U,\Hom(\mathcal E,\cdot))$,
where $\mathcal E$ is a free sheaf.
Let $$A_k:=\Gamma(U,\Hom(\mathcal E,\mathcal A_k)), \quad (0 \leq k \leq M),$$ $$F_k:=\Gamma(U,\Hom(\mathcal E,\mathcal F_{1,k})), \quad (1 \leq k \leq M).$$
Then we obtain commutative diagrams of abelian groups
\begin{equation}
\label{diag1}
\bfig
\node a12(0,0)[\dots]
\node a22(-700,0)[A_{M}]
\node a32(700,0)[A_2]
\node a42(1400,0)[A_1]
\node a52(2100,0)[A_0]
\node a13(0,-500)[\dots]
\node a23(-700,-500)[F_{M}]
\node a33(700,-500)[F_2]
\node a43(1400,-500)[F_1]
\node b1(-700,500)[0]
\node b2(700,500)[0]
\node b3(1400,500)[0]
\arrow[a22`a12;a_{M-1}]
\arrow[a12`a32;a_2]
\arrow[a32`a42;a_1]
\arrow[a42`a52;a_0]
\arrow[a52`a62;]
\arrow[a23`a13;p_{M-1}]
\arrow[a13`a33;p_2]
\arrow[a33`a43;p_1]
\arrow[a23`a22;w_{M}]
\arrow[a33`a32;w_2]
\arrow[a43`a42;w_1]
\arrow[a22`b1;]
\arrow[a32`b2;]
\arrow[a42`b3;]
\efig
\end{equation}
and
\begin{equation}
\label{s5}
F_{M} \oplus \Ker w_{M-1} \overset{b_{M-1}}{\to} \dots \overset{b_2}{\to} F_2 \oplus \Ker w_{1} \overset{b_1}{\to} F_1 \overset{b_0}{\rightarrow} A_0 \to 0.
\end{equation}
Note that complete exactness of sequence (\ref{s4}) truncated to term $\mathcal F_{1,M-2} \oplus \Ker \omega_{M-3}$ is equivalent by definition to the exactness of  sequence (\ref{s5}) truncated to term $F_{M-2} \oplus \Ker w_{M-3}$.
By Definition \ref{adef2} the middle row of (\ref{diag1}) is exact. 
Also, by Proposition \ref{vanprop}(1) each $w_k$, $1 \leq k \leq M$, is surjective, so the columns of (\ref{diag1}) are exact. Hence, we have analogous identities
\begin{equation}
\label{s6}
b_0=a_0w_1, \quad b_1=p_1-i_1, \quad b_k=(i_k\oplus p_{k-1})(p_k-i_k),
\end{equation}
where $i_k:\Ker w_k \hookrightarrow F_k$ is an inclusion.
Let us show that (\ref{s5}) is exact up to term $F_{M-2} \oplus \Ker w_{M-3}$. First, note that $b_0$ is surjective because both $a_0$ and $w_1$ are. Second, if $\xi \in \Ker b_0$, then $w_1(\xi) \in \Ker a_0=\Imag a_1=\Imag a_1w_2=\Imag w_1p_1$.
Here we have used the fact that $w_2$ is surjective. Let $w_1(\xi)=w_1(p_1(\zeta))$ and $\tau:=p_1(\zeta)-\xi \in \Ker w_1$, so that $\xi=b_1(\zeta,\tau) \in \Imag b_1$.
Third, if $1 \leq k \leq M-3$, and $(\xi,\eta) \in \Ker b_k=\Ker (p_k-i_k)$, then $\eta=p_k(\xi)$ and $0=w_k(p_k(\xi))=a_k(w_{k+1}(\xi))$; hence $w_{k+1}(\xi) \in \Imag a_{k+1}=\Imag a_{k+1}w_{k+2}=\Imag w_{k+1}p_{k+1}$. Choose $\zeta$ so that $w_{k+1}(\xi) =w_{k+1}(p_{k+1}(\zeta))$. Then $\tau:=p_{k+1}(\zeta)-\xi \in \Ker w_{k+1}$. We conclude that $(\xi,\eta)=b_{k+1}(\zeta,\tau) \in \Imag b_{k+1}$, i.e., sequence (\ref{s5}) is indeed exact up to term $F_{M-2} \oplus \Ker w_{M-3}$.

Now, by Lemma \ref{lem55} each $\mathcal F_{1,k} \oplus \Ker \omega_{k-1}$ has a free resolution of length $N+n-1$. Hence, if we take
\begin{equation*}
\mathcal B_1:=\mathcal F_{1,1}, \quad \varepsilon_0:=\beta_0, \quad \mathcal B_k:=\mathcal F_{1,k} \oplus \Ker \omega_{k-1}, \quad 2 \leq k \leq M-2,
\end{equation*}
we obtain the required completely exact sequence (\ref{seq2}).

\medskip

(2) Now, consider completely exact sequence obtained from (\ref{seq2}),
\begin{equation*}
\mathcal B_M  \overset{\beta_{M-1}}{\to}  \dots \overset{\beta_2}{\to} \mathcal B_2 \overset{\beta_1}{\to} \Ker \varepsilon_0 \to 0.
\end{equation*}
Applying case (1) to this sequence we obtain that there is a completely exact sequence
\begin{equation*}
\mathcal D_{M-4}  \to \dots \to \mathcal D_3 \to \mathcal D_2 \overset{\varepsilon_1}{\to} \Ker \varepsilon_0  \to 0,
\end{equation*}
where $\mathcal D_2$ is a free sheaf, and each sheaf $\mathcal D_k$, $3 \leq k \leq M-4$, has a free resolution of length $N+n-2$. Therefore, we have a completely exact sequence
\begin{equation*}
\mathcal D_{M-4}  \to \dots \to \mathcal D_3 \to \mathcal D_2 \overset{\varepsilon_1}{\to} \mathcal B_1 \to \mathcal A_0\to 0.
\end{equation*}
Continuing in this way (applying a similar argument to resolution of $\Ker \varepsilon_1$, etc.) we finally obtain a free resolution of $\mathcal A_0$ of length $N$. 
\end{proof}

\begin{proof}[Proof of Lemma \ref{lem7}]
We can assume that $\mathcal A \subset \mathcal B$ and that $\beta$ is the inclusion map.

(a) If $\mathcal A$ and $\mathcal B$ have free resolutions of length $N+n$, then Lemma \ref{lem6} implies that $\mathcal C$ has a free resolution of length $N$ (and, in particular, of length $N-n-1$).

Consider two remaining cases. Sheaf $\mathcal C$ has a free resolution of length $N+n$, 
\begin{equation}
\label{s9}
\mathcal F_{N+n} \to \dots \to \mathcal F_1 \overset{\varphi}{\to} \mathcal C \to 0
\end{equation}
for some open $V_0 \subset U_0$.
By Proposition \ref{vanprop}(1) sequence (\ref{s9}) is completely exact. By Lemma \ref{lem65} there is a commutative diagram
\begin{equation}
\label{s10}
\bfig
\node a1(0,0)[0]
\node a2(500,0)[\mathcal A]
\node a3(1000,0)[\mathcal B]
\node a4(1500,0)[\mathcal C]
\node a5(2000,0)[0]
\node b4(1500,-500)[\mathcal F_1]
\node c(1500,500)[0]
\arrow[a1`a2;]
\arrow[a2`a3;\beta]
\arrow[a3`a4;\gamma]
\arrow[a4`a5;]
\arrow[b4`a4;\varphi]
\arrow[b4`a3;\psi]
\arrow[a4`c;]
\efig.
\end{equation}
Let $\iota:\Ker \varphi \rightarrow \mathcal F_1$ denote the inclusion. Let us show that the sequence 
\begin{equation}
\label{s11}
0 \to \Ker \varphi \overset{\psi \oplus \iota}{\to} \mathcal A \oplus \mathcal F_1 \overset{\beta-\psi}{\to} \mathcal B \to 0
\end{equation}
is completely exact.

We apply functor $\Gamma(U,\Hom(\mathcal E,\cdot))$ to (\ref{s10}) and (\ref{s11}), where $\mathcal E$ is a free sheaf. We obtain diagrams of abelian groups
\begin{equation*}
\bfig
\node a1(0,0)[0]
\node a2(500,0)[A]
\node a3(1000,0)[B]
\node a4(1500,0)[C]
\node a5(2000,0)[0]
\node b4(1500,-500)[F_1]
\node c(1500,500)[0]
\arrow[a1`a2;]
\arrow[a2`a3;b]
\arrow[a3`a4;c]
\arrow[a4`a5;]
\arrow[b4`a4;f]
\arrow[b4`a3;p]
\arrow[a4`c;]
\efig
\end{equation*}
and
\begin{equation}
\label{s12}
0\to \Ker f \overset{p \oplus i}{\to} A \oplus F_1 \overset{b-p}{\to} B \to 0.
\end{equation}
The first diagram is commutative, its top row is exact (see~Definition \ref{adef2}). By Proposition \ref{vanprop}(1) we may assume that $f$ is surjective. The latter sequence is a complex and is exact at $\Ker f$. We have to check that it is exact at the next two terms. If $(\xi,\eta) \in \Ker(b-p)$ then $p(\eta)=b(\xi)=\xi$, $0=c(p(\eta))=f(\eta)$. Thus, $\eta \in \Ker f$ and $(\xi,\eta)=(p \oplus i)(\eta) \in \Imag (p \oplus i)$, hence (\ref{s12}) is exact in the middle term. On the other hand, if $\zeta \in B$, then with some $\eta \in F$
\begin{equation*}
-c(\zeta)=f(\eta)=c(p(\eta)), 
\end{equation*}
i.e.,
\begin{equation*}
\zeta+p(\eta)=\xi \in \Ker c=A.
\end{equation*}
Thus, $\zeta=\xi-p(\eta) \in \Imag(b-p)$. We obtain that sequence (\ref{s12}) is exact, hence sequence (\ref{s11}) is completely exact. 

Note that by Lemma \ref{lem55} $\Ker \varphi$ has a free resolution of length $N+n-1$.

\medskip

(b) The sheaf $\mathcal A \oplus \mathcal F_1$ has a free resolution of length $N+n-1$ over $U_0 \times K$. By Lemma \ref{lem6} sheaf $\mathcal B$ has a free resolution of length $N-1$ over $U_0 \times K$ (in particular, of length $N-n-1$). 

\medskip

(c) 
We may assume that $\mathcal B$ has a free resolution of length $N+n-1$ over $V_0 \times K$.
Since  $\Ker \varphi$ has a free resolution of length $N+n-1$, 
by (b) $\mathcal A \oplus \mathcal F$ has a free resolution of length $N-1$. Since sequence $0 \rightarrow \mathcal F \rightarrow \mathcal A \oplus \mathcal F \rightarrow \mathcal A \rightarrow 0$
is completely exact as $\mathcal F$ is a free sheaf (see~Corollary \ref{cechlem1}), we obtain by part (a) that $\mathcal A$ has a free resolution of length $N-n-1$.
\end{proof}

\addtocontents{toc}{\protect\setcounter{tocdepth}{1}}

\end{document}